%% file: main.tex
\documentclass{article}

\usepackage[total={6.5in,8.75in},
top=1.2in, left=0.9in, includefoot]{geometry}
\usepackage[utf8]{inputenc} 
\usepackage[T1]{fontenc} 
\usepackage{microtype}
\usepackage[makeroom]{cancel}
\usepackage{booktabs} 
\usepackage[colorlinks=true,citecolor=blue]{hyperref}
\usepackage{multirow}
\usepackage[table, svgnames, dvipsnames]{xcolor}
\usepackage{makecell, cellspace, caption}
\usepackage{booktabs,caption}
\usepackage[flushleft]{threeparttable}
\usepackage{nccmath}
\usepackage{amsmath,amssymb,amsfonts,amsthm}
\usepackage{mathtools}
\usepackage{enumitem}
\usepackage{amsthm} 
\usepackage{lipsum}
\usepackage[capitalize,noabbrev]{cleveref}
\usepackage{cite}
\usepackage{graphicx} 
\usepackage{hyperref}       
\usepackage{url}          
\usepackage{nicefrac}      
\usepackage{microtype}  
\usepackage{algorithm}
\usepackage{algorithmic}
\usepackage{textcomp}
\usepackage{color}
\usepackage{pifont}
\usepackage{microtype}
\usepackage{subcaption}
\usepackage[makeroom]{cancel}
\usepackage[colorlinks=true,citecolor=blue]{hyperref}
\usepackage[table, svgnames, dvipsnames]{xcolor}
\usepackage{booktabs,caption}
\usepackage[flushleft]{threeparttable}
\usepackage{mathtools}
\usepackage[capitalize,noabbrev]{cleveref} 
\usepackage{dsfont}
\usepackage{titlesec}
\titlespacing*{\paragraph}
  {0pt}    
  {1ex}    
  {1em}    

\newtheorem{theorem}{Theorem}
 
\newtheorem{lemma}{Lemma}
\newtheorem{assumption}{Assumption}

\newtheorem{claim}{Claim}

\theoremstyle{remark}
\newtheorem{remark}{Remark}

\newcommand{\citet}{\cite}
\newcommand{\citep}{\cite}
\usepackage{xcolor}
\newcommand{\reb}[1]{\textcolor{black}{#1}}

\input{macros}

\input{math_commands}

\title{Decentralized Nonconvex Optimization under Heavy-Tailed Noise: Normalization and Optimal Convergence}
\author{Shuhua Yu\thanks{
     Department of Electrical and Computer Engineering,
     Carnegie Mellon University, Pittsburgh, PA 15213, USA. Emails:
     \texttt{\{shuhuay, soummyak\}@andrew.cmu.edu}}
     \and
     Du\u san Jakoveti\'c\thanks{
     Faculty of Sciences, Department of Mathematics and Informatics, University of Novi Sad, Novi Sad, 21000, Serbia. Email: \texttt{dusan.jakovetic@dmi.uns.ac.rs}}
     \and
     Soummya Kar\footnotemark[1]} 
\date{April 2026}

\begin{document}
\maketitle
\vskip 0.3in 

\input{abstract}
\input{intro}
\input{problem}

\input{alg} 
\input{theorems}
\input{exp}
\input{conclusion}
\input{ack}

\bibliographystyle{ieeetr}
\bibliography{refs}

\appendix
\section*{Appendix}
\input{proof_claim} 
\input{analysis}
\input{exp_add}

\end{document}

%% file: macros.tex
\def\R{\mathbb{R}}
\def\E{\mathbb{E}}
\def\N{\mathbb{N}}
\def\P{\mathbb{P}}

\newcommand{\ba}{\boldsymbol{a}}
\newcommand{\bb}{\boldsymbol{b}}

\newcommand{\bd}{\boldsymbol{d}}

\newcommand{\g}{\boldsymbol{g}}

\newcommand{\m}{\boldsymbol{m}}

\newcommand{\bs}{\boldsymbol{s}}

\newcommand{\bv}{\boldsymbol{v}}
\newcommand{\w}{\boldsymbol{w}}
\newcommand{\x}{\boldsymbol{x}}
\newcommand{\y}{\boldsymbol{y}}
\newcommand{\z}{\boldsymbol{z}}

\def\bA{\boldsymbol{A}}

\def\bI{\boldsymbol{I}}

\def\bL{\boldsymbol{L}}

\def\bW{\boldsymbol{W}}
\def\bX{\boldsymbol{X}}

\def\cE{\mathcal{E}}
\def\cF{\mathcal{F}}
\def\cG{\mathcal{G}}

\def\cN{\mathcal{N}}

\def\cV{\mathcal{V}}

\DeclareMathOperator{\clip}{\textsf{clip}}

\newcommand{\bxi}{\boldsymbol{\xi}}
 
\newcommand{\beps}{\boldsymbol{\epsilon}}

\newcommand{\bPsi}{\boldsymbol{\Psi}}

\newcommand{\one}{\mathbf{1}}
\newcommand{\zero}{\mathbf{0}}
\newcommand{\ox}{\bar{\x}}
\newcommand{\ov}{\bar{\bv}}
\newcommand{\oy}{\bar{\y}} 
\newcommand{\ts}{\tilde{\bs}}
\newcommand{\tz}{\tilde{\z}} 
\newcommand{\tbeps}{\tilde{\beps}}
\newcommand{\onab}{\overline{\nabla}}
\newcommand{\ty}{\tilde{\y}}

\newcommand{\dsgd}{\texttt{DSGD}}
\newcommand{\dsgdgclip}{\texttt{DSGD-Clip}}
\newcommand{\gtdsgd}{\texttt{GT-DSGD}}
\newcommand{\gtnsgdm}{\texttt{GT-NSGDm}}
\newcommand{\sen}{\texttt{SClip-EF-Network}}

\newcommand{\argmin}{\mathop{\mathrm{arg\,min}}}



%% file: math_commands.tex
\usepackage{amsmath, amsfonts, bm}

\def\1{\bm{1}}

\DeclareMathAlphabet{\mathsfit}{\encodingdefault}{\sfdefault}{m}{sl}
\SetMathAlphabet{\mathsfit}{bold}{\encodingdefault}{\sfdefault}{bx}{n}



%% file: abstract.tex
\begin{abstract}
Heavy-tailed noise in nonconvex stochastic optimization has garnered increasing research interest, as empirical studies, including those on training attention models, suggest it is a more realistic gradient noise condition. This paper studies first-order nonconvex stochastic optimization under heavy-tailed gradient noise in a decentralized setup, where each node can only communicate with its direct neighbors in a predefined graph. Specifically, we consider a class of heavy-tailed gradient noise that is zero-mean and has only $p$-th moment for $p \in (1, 2]$. We propose $\gtnsgdm$, Gradient Tracking based Normalized Stochastic Gradient Descent with momentum, that utilizes normalization, in conjunction with gradient tracking and momentum, to cope with heavy-tailed noise on distributed nodes. We show that, when the communication graph admits primitive and doubly stochastic weights, $\gtnsgdm$\ guarantees, for the \textit{first} time in the literature, that the expected gradient norm converges at an \textit{optimal non-asymptotic rate} $O\big(1/T^{(p-1)/(3p-2)}\big)$, which matches the lower bound in the centralized setup. When the tail index $p$ is unknown, \gtnsgdm\ attains a non-asymptotic rate $O\big( 1/T^{(p-1)/(2p)} \big)$ that is, for $p < 2$, topology independent and has a speedup factor $n^{1-1/p}$ in terms of the number of nodes $n$. Finally, experiments on nonconvex linear regression with tokenized synthetic data and decentralized training of language models on a real-world corpus demonstrate that \gtnsgdm\ is more robust and efficient than  baselines.
\end{abstract} 

%% file: intro.tex
\section{Introduction}

In this paper, we address the problem of nonconvex stochastic optimization under heavy-tailed gradient noise in the decentralized setup. Consider a graph with $n$ nodes connected by a predefined topology $\cG: = (\cV, \cE)$, where $\cV := \{1, \ldots, n\}$ is the set of node indices, and $\cE$ is the collection of directed pairs $(i, r)$, $i, r \in \cV$ such that node $i$ can send information to the neighboring node $r$. Each node $i \in \cV$ holds a local nonconvex differentiable cost function $f_i: \R^d \rightarrow \R$, and can access its stochastic gradient, subject to zero mean noise with a bounded $p$-th moment for some $p \in (1, 2]$. Cooperatively, these nodes aim to solve $\min_{\x \in \R^d} f(\x) := (1/n)\sum_{i=1}^n f_i(\x)$, through local computation and peer-to-peer communication only with their immediate neighbors.




Decentralized optimization in the above formulation has been studied for decades \citep{tsitsiklis1986distributed}, and has recently attracted growing research interest due to its advantages in scalability and privacy preservation across a wide range of distributed machine learning, signal processing, and control tasks over networks \citep{nedic2018network, li2020federatedspm, kairouz2021advances}. For instance, in privacy-sensitive applications such as those in the medical domain \citep{brisimi2018federated}, training data are often distributed across $n$ nodes due to privacy constraints. In such cases, each $f_i$ represents an empirical risk function, e.g., a neural network, defined over the local dataset on node $i$, and all nodes collaboratively train a global predictive model via peer-to-peer communication without sharing raw data. Moreover, decentralized optimization is also employed in data centers to reduce communication bottlenecks associated with the central node in traditional centralized training paradigms \citep{lian2017can}.


In decentralized optimization, first-order methods are widely favored for their simplicity and scalability \citep{xin2020general}. However, computing the exact gradient of each local objective function $f_i$ at every iteration can be computationally expensive, particularly in large-scale settings where each node holds a substantial volume of local data. To alleviate this computational burden, decentralized stochastic gradient methods, which approximate exact gradients, have been extensively studied. Most existing approaches, including decentralized stochastic (sub)gradient descent \citep{sundhar2010distributed, koloskova2020unified, wang2021cooperative}, variance reduction techniques \citep{yuan2018variance}, and gradient tracking-based schemes \citep{di2016next, pu2021distributed}, typically assume that stochastic gradient noise has a \textit{finite variance}. Nevertheless, recent empirical and theoretical evidence indicates that, when optimizing certain neural network architectures, especially attention-based models such as Transformers \citep{vaswani2017attention}, the gradient noise often follows a \textit{heavy-tailed distribution}\footnote{A random variable $X$ is called heavy-tailed if it exhibits a heavier tail than any exponential distribution; formally, for any constant $a > 0$, $\limsup_{x \rightarrow \infty} \P(X > x)e^{ax} = \infty$ \citep{nair2022fundamentals}. While some heavy-tailed distributions, such as log-normal and Weibull, still have bounded variance, this paper also considers the sub-class of heavy tailed gradient noise that may have unbounded (infinite) variance such as $\alpha$-stable noise.} with significantly large or even \textit{infinite variance} \citep{simsekli2019tail, zhang2020adaptive, gorbunov2020stochastic, gurbuzbalaban2021heavy, ahn2024lineartransformer, kunstner2024heavy}. The presence of heavy-tailed gradient noise poses substantial challenges for existing methods. Empirically, some stochastic gradient descent (SGD) based methods can suffer from instability and even dramatic drop of training accuracies \citep{zhang2020adaptive, charles2021large, yang2022taming}, particularly in distributed large-cohort training. Theoretically, unbounded variance renders many established analyses invalid, and in \textit{centralized} settings  it necessitates the use of nonlinear adaptive techniques such as clipping, sign, and normalization  \citep{zhang2020adaptive, sadiev2023high, compagnoni2025adaptive, hubler2024gradient, liu2025nonconvex, armacki2025high} to combat the strong noise. However, incorporating such adaptive strategies in \textit{decentralized} algorithms introduces inherent \textit{nonlinearity} into the algorithmic dynamics associated with the average-sum structured function $f$, making the design and analysis of decentralized algorithms under heavy-tailed noise significantly more challenging.


Decentralized optimization under heavy-tailed gradient noise remains underexplored. To the best of our knowledge, only recent studies \citet{sun2024distributed,yu2023smoothed} have attempted to address this problem under restrictive assumptions. Specifically, \citet{sun2024distributed} considers zero-mean gradient noise with bounded $p$-th central moment ($p \in (1,2]$) similar to our setting but assumes a \textit{compact} domain or \textit{bounded gradients}. Their proposed decentralized gradient descent method with $\ell_2$ gradient clipping achieves almost sure convergence for \textit{strongly convex} local functions. However, the restrictive compact domain or gradients assumption in \citet{sun2024distributed} limits its practical applicability, and the convergence rate is not explicitly provided. Another work, \citet{yu2023smoothed}, also assumes strongly convex local objectives and develops a decentralized gradient method with smoothed clipping and error feedback under gradient noise that is zero-mean, \textit{symmetric}, and has bounded first absolute moment, showing an \textit{in-expectation} convergence rate of $1/t^{\delta}$ for some $\delta \in (0, 2/5)$. Although the noise assumption in \citet{yu2023smoothed} is weaker than ours (as it requires only a first-moment bound), the additional assumptions of \textit{noise symmetry} and the dependence of the rate exponent $\delta$ on both the problem dimension and condition number restrict its general applicability. Moreover, both works \citet{sun2024distributed,yu2023smoothed} assume strong convexity, whereas many practical optimization problems involving heavy-tailed noise, particularly in modern machine learning, are inherently nonconvex. Further, the convergence rates in \citet{sun2024distributed,yu2023smoothed} are either unclear or sub-optimal, even compared to the optimal iteration complexity bound $O\big(1/T^{(p-1)/(3p-2)}\big)$ for general nonconvex functions. In this work, we relax these restrictive assumptions and address the following question:

\begin{center}
\textit{Can we design a decentralized algorithm for \textbf{nonconvex} optimization under general zero-mean gradient noise with only a finite $p$-th moment for $p \in (1, 2]$ with \textbf{optimal iteration complexity}?}
\end{center}

\subsection{Contributions}
We answer this question affirmatively through the following key contributions:
\begin{itemize}
    \item We develop a decentralized method, called \gtnsgdm, using normalization, coupled with momentum variance reduction, to combat heavy-tailed noise, and using gradient tracking to handle cross-node heterogeneity. To further shed light on the design of \gtnsgdm, we provide a negative result for a vanilla variant of normalized decentralized SGD that employs neither gradient tracking nor momentum. 

    \item For general nonconvex and smooth local functions $f_i$'s that are bounded from below, we show that \gtnsgdm\ converges in expectation at a rate $O\big(1/T^{(p-1)/(3p-2)}\big)$, which matches the lower bound in the centralized setting and is order-optimal. Our convergence rate significantly improves upon related works \citep{sun2024distributed, yu2023smoothed}, which assume strong convexity and lack an explicit rate exponent.

    \item When the tail index $p$ is unknown, \gtnsgdm\ achieves a rate of $O(1/T^{(p-1)/(2p)})$, matching the best-known rate in the centralized setting without requiring knowledge of $p$. Notably, for $p \in (1, 2)$ and sufficiently large $T$, this rate is \textit{independent} of the network topology and exhibits a \textit{speedup} in the number of nodes, with a factor of $n^{1 - 1/p }$. 

    \item We test our theoretical findings in nonconvex linear regression models on a synthetic dataset that is built to simulate language tokens under controlled heavy-tailed noise injections. We also test \gtnsgdm\ on distributed training of decoder-only Transformer models on Multi30k datasets \citep{elliott2016multi30k}. Experiments on multiple variants of network topologies show that \gtnsgdm\ is more robust to injected and empirical heavy-tailed noise and converges faster. 
\end{itemize}

\subsection{Related Work}

\textbf{Heavy-tailed gradient noise.} Recent empirical studies suggest that the distribution of gradient noise in training various deep learning models resembles heavy-tailed distributions, such as Lévy’s $\alpha$-stable distribution \citep{simsekli2019tail, zhang2020adaptive, barsbey2021heavy, battash2024revisiting}. For instance, the work \citet{zhang2020adaptive} demonstrates that the empirical distribution of gradient norm samples during BERT pre-training closely aligns with an $\alpha$-stable distribution, rather than a Gaussian one (see their Figure 1). The presence of heavy-tailed gradient noise is also supported by theoretical insights \citep{simsekli2019tail, peluchetti2020stable, gurbuzbalaban2021heavy, barsbey2021heavy}. In particular, \citet{simsekli2019tail} leverages generalized central limit theorems to show that the gradient noise in SGD can converge to an $\alpha$-stable random variable.


\textbf{Adaptive methods.} Under heavy-tailed noise, vanilla SGD based methods are shown to suffer from slower convergence or even model collapses in \textit{centralized} settings \citep{zhang2020adaptive} as well as \textit{distributed settings with a central server} \citep{yang2022taming, lee2025efficient}, and adaptive methods such as clipping and normalization are introduced to stabilize training dynamics. In \textit{centralized} settings, the work \citet{zhang2020adaptive} provides lower bounds for both nonconvex and strongly convex smooth functions, showing that SGD with gradient clipping achieves \textit{in-expectation} upper bounds matching lower bounds. In \citet{sadiev2023high, liu2023breaking, nguyen2023improved, chezhegov2024gradient}, the authors show that when equipped with gradient clipping, SGD, accelerated methods, AdaGrad \citep{duchi2011adaptive}, and Adam \citep{kingma2014adam} can achieve (near-)optimal \textit{high-probability} convergence under various function assumptions. Besides, the work \citet{compagnoni2025adaptive} shows that signSGD is also robust to heavy-tailed noise through the lens of stochastic differential equations. Further, SGD with gradient normalization, which advantageously requires less hyper-parameter tuning than clipping, is shown to achieve optimal \textit{in-expectation} convergence \citep{hubler2024gradient, liu2025nonconvex, sun2024gradient}. Our method incorporates the same normalization and variance reduction approach as \citet{liu2025nonconvex}. Notably, in another line of works \citet{jakovetic2023nonlinear, armacki2025high, armacki2024large}, the authors conduct a unified convergence analysis for generic nonlinear methods including clipping, sign, and normalization under \textit{symmetric} noise with positive probability mass around zero without assuming any noise moment bound or only assuming a first absolute noise moment bound. In \textit{distributed settings with a server}, the work \citet{gorbunov2024highprobability} proposes an algorithm that incorporates an error feedback mechanism, wherein clipping is applied to the discrepancy between a local gradient estimator and a stochastic gradient, and establishes optimal \textit{high-probability} bounds. Moreover, the work \citet{compagnoni2025unbiased} shows that distributed signSGD converges to an asymptotic neighborhood depending on the `fatness' of the noise tail. When multiple local updates are permitted between communication rounds, the authors of \citet{yang2022taming} show that clipping per local step achieves order-optimal in-expectation convergence, albeit under a restrictive \textit{bounded gradient} assumption. \reb{In addition, the work \citet{lee2024efficient} rigorously shows that distributed training under heavy-tailed noise can collapse with unbounded gradients, and shows that client-side adaptive methods can stabilize training with bounded regrets in an online learning framework.} More recently, the work \citet{lee2025efficient} introduces the TailOPT framework, which adaptively leverages gradient geometry by applying clipping operators during local updates on distributed nodes and utilizing adaptive optimizers for global updates at the server, achieving in-expectation sublinear convergence rates that are independent of the moment parameter $p$.

\textbf{Nonlinearities in decentralized optimization.} Extending existing methods that are robust to heavy-tailed noise, whether developed for \textit{centralized} settings or \textit{distributed settings with a server}, to decentralized environments is highly nontrivial, primarily due to the \textit{nonlinearities} introduced to peer-to-peer communication. This difficulty is reflected in that existing decentralized methods incorporating nonlinear adaptive techniques for other purposes often impose restrictive conditions \citep{yu2023secure, li2025convergence}. For example, to achieve differential privacy through gradient clipping, the work \citet{li2025convergence} establishes convergence in decentralized setups under the assumption of either \textit{bounded gradient} or a \textit{stringent similarity} condition, namely $\| \nabla f_i(\x) - \nabla f(\x) \| \le (1/12) \|\nabla f(\x)\|$ for all $i \in [n]$ and all $\x$. Similarly, to attain adversarial robustness against gradient attacks, the authors of \citet{yu2023secure} employ gradient clipping with momentum, assuming that all local functions are convex, share a \textit{common minimizer}, and that $\sum_{i=1}^n f_i$ is strongly convex. \reb{Furthermore, \citet{taheri2023generalization} integrates gradient tracking with normalization to address a max-margin problem. While gradient tracking is employed to ensure similarity among local gradient estimators, the method guarantees convergence only to the direction of the optimal solution under deterministic gradients.} In this work, we significantly relax these conditions and demonstrate the effective use of nonlinearity (specifically, normalization) in decentralized optimization, thereby motivating broader  applications of nonlinear techniques in this setting.  


\subsection{Notation}

We denote by $\N_+$, $\R$, $\R_+$ and $\R^d$, respectively, the set of positive natural numbers, real numbers, nonnegative real numbers, and the $d$-dimensional Euclidean space. We use lowercase normal letters for scalars, lowercase boldface letters for vectors, and uppercase boldface letters for matrices. Further, we denote by $\one_k$ and $\zero_k$ the all-ones and all-zeros vectors of size $k$, respectively, and by $\bI_k$ the $k \times k$ identity matrix. We let $\|\x\|$ denote the $\ell_2$ norm of $\x$, and $\|\bA\|_2$ denote the operator norm of $\bA$. For functions $p(t)$ and $q(t)$ in $t$, we write $p(t) = O(q(t))$ if $\limsup_{t \to \infty} p(t)/q(t) < \infty$. Finally, we use $\E$ to denote expectation over random quantities. 

%% file: problem.tex
\section{Problem Formulation} 
\label{sec:problem}
We consider a graph with $n$ nodes, where each node holds a local and private function $f_i: \R^d \rightarrow \R$, and the nodes collectively minimize the unconstrained global objective $f(\x) := (1/n)\sum_{i=1}^n f_i(\x)$ through peer-to-peer communication. We now present some standard assumptions on the problem.

\begin{assumption}[Finite lower bound]
\label{as:ldd} There exists some 
$f_* := \inf_{\x \in \R^d} f(\x) > -\infty$. 
\end{assumption}

\begin{assumption} [$L$-smoothness] 
\label{as:smooth} 
The local function $f_i$ at each node $i \in [n]$ is differentiable and $L$-smooth, i.e., $\forall \x, \y \in \R^d, 
\| \nabla f_i(\x)  - \nabla f_i(\y) \| \le L \| \x - \y \|$.
\end{assumption} 

We next introduce the heavy-tailed noise model. For each node $i \in \cV$, at $t$-th iteration with query $\x_i^t$, the stochastic first-order oracle returns the gradient estimator $\g_i(\x_i^t, \bxi_i^{t})$, where $ \bxi_i^{t}$ denotes the random sample. Let $\Omega, \emptyset$ denote the universe, empty set, respectively. We use the following natural filtration, i.e., an increasing family of sub-$\sigma$-algebras, to denote the past history up to iteration $t$: 
\begin{align*}
    \cF_{-1} := \{\Omega, \emptyset\}, \quad \cF_{t} := \sigma \big( \big \{  \bxi_i^0, \ldots, \bxi_i^{t - 1}  : i \in [n] \big \} \big), \forall t \ge 0. 
\end{align*} 
We then assume this stochastic first-order oracle has the following properties. 

\begin{assumption}[Heavy-tailed noise]
\label{as:noise}
For any $\cF_t$-measurable random vectors $\x \in \R^d$, we have the following: $\forall i \in [n], \forall t \ge 0$, (1) $\E [ \g_i(\x, \bxi_i^t) \mid \cF_t ] = \nabla f_i(\x)$; (2) There exist $p \in (1, 2]$, some constant $\sigma \ge 0$ such that $\E \big[ \| \g_i(\x, \bxi_i^{t}) - \nabla f_i(\x) \|^p \mid \cF_{t} \big] \le \sigma^p$; (3) The family $\{ \bxi_i^t: \forall t \ge 0, i \in [n]\}$ of random samples is independent. 
\end{assumption}

\begin{remark}[Heavy-tailed distributions] 
Assumption \ref{as:noise} covers a broad class of heavy-tailed distributions, including Lévy’s $\alpha$-stable distributions, Student’s $t$-distributions, and Pareto distributions. Note that we do not assume noise symmetry as in \citet{yu2023smoothed}, and when $p = 2$, Assumption \ref{as:noise} reduces to the standard bounded variance condition commonly assumed in the literature.
\end{remark} 

\begin{figure}[ht]
\centering 
\begin{subfigure}[b]{0.23\textwidth}
    \includegraphics[width=\linewidth]{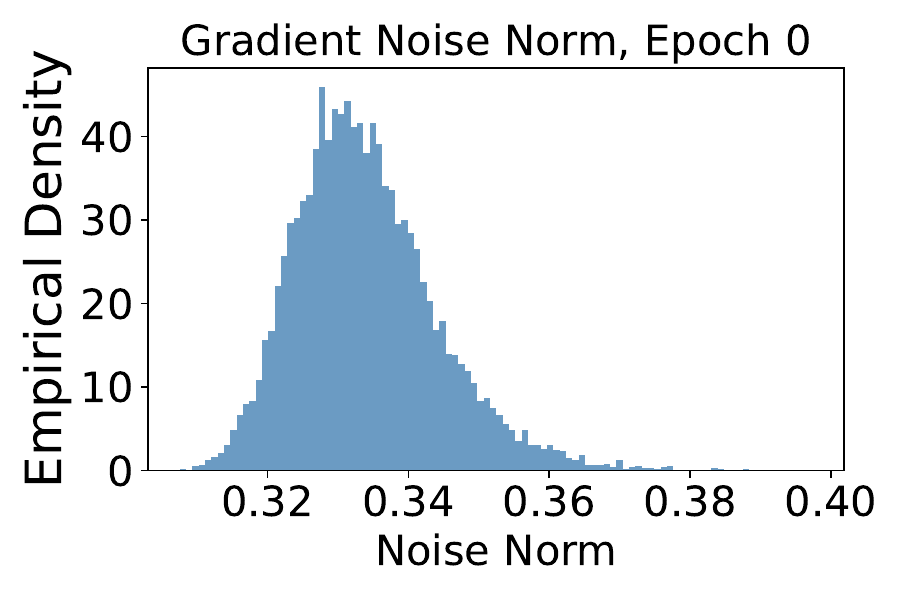}
    \caption{Epoch 0}
\end{subfigure}
\hfill
\begin{subfigure}[b]{0.23\textwidth}
    \includegraphics[width=\linewidth]{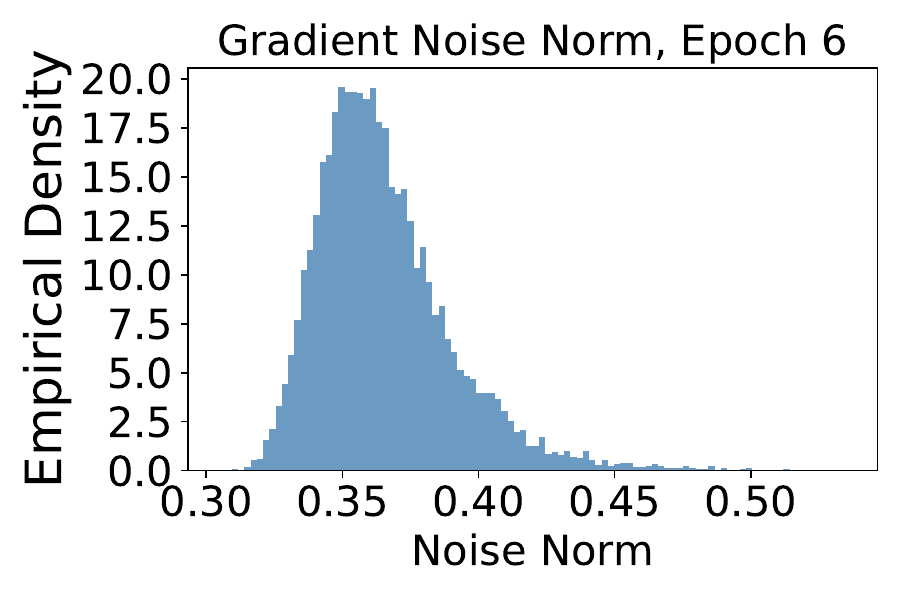}
    \caption{Epoch 6}
\end{subfigure}
\hfill
\begin{subfigure}[b]{0.23\textwidth}
    \includegraphics[width=\linewidth]{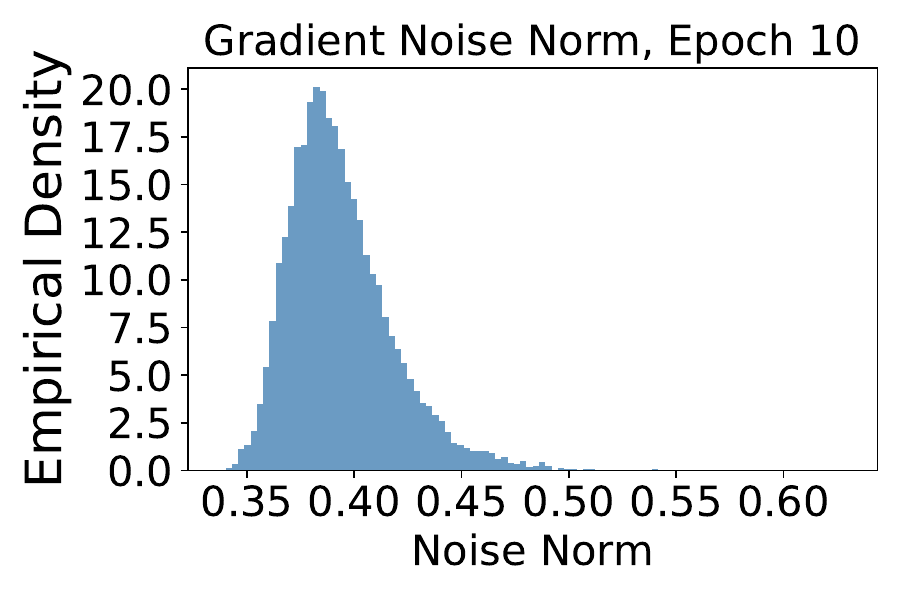}
    \caption{Epoch 10}
\end{subfigure}
\hfill
\begin{subfigure}[b]{0.23\textwidth}
    \includegraphics[width=\linewidth]{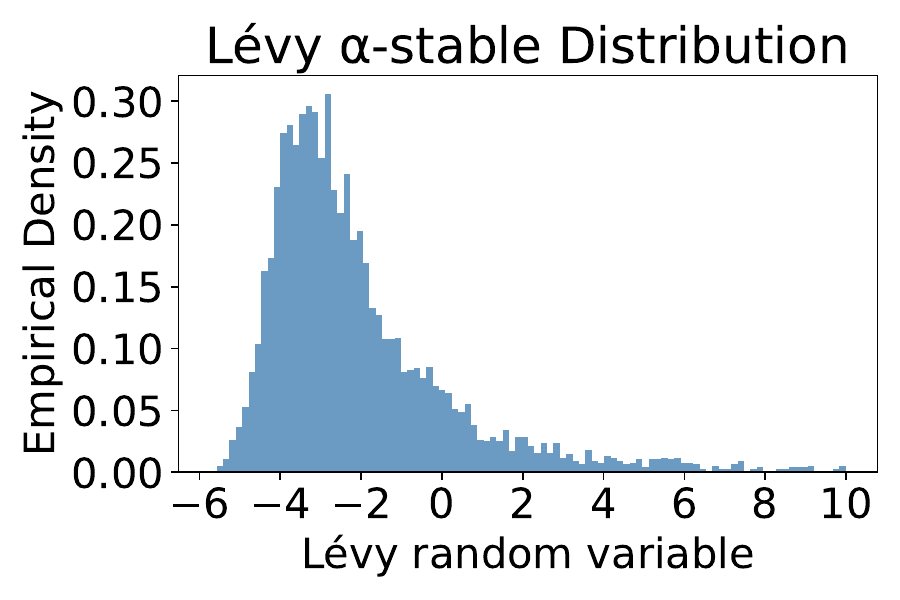}
    \caption{$\alpha$-stable distribution}
\end{subfigure}

\caption{Comparisons of the empirical density of gradient noise norm in different epochs of training a Transformer model with a synthetic Lévy $\alpha$-stable distribution.}
\label{fig:empirical-evidence}
\end{figure}

\begin{remark}[Empirical evidence] Similar to \citet{zhang2020adaptive, yang2022taming}, we investigate the empirical distribution of the gradient noise norm $\| \g(\x, \bxi) - \nabla f(\x) \|$ in a centralized setting by training a GPT model \citep{radford2018improving} with 3M parameters on the Multi30k dataset \citep{elliott2016multi30k}, where $\g(\x)$ denotes the mini-batch stochastic gradient and $\nabla f(\x)$ denotes the full-batch gradient. We train the model for 12 epochs using SGD and plot the empirical density of the noise norm at the beginning of epochs 0, 6, and 10. As shown in Figure \ref{fig:empirical-evidence}, as training progresses, the tail of the empirical gradient noise norm distribution becomes heavier (and longer) and increasingly resembles that of a synthetic $\alpha$-stable distribution.
\end{remark}

For peer-to-peer communication in decentralized settings, we need to specify a mixing matrix $\bW$ on graph $\cG = (\cV, \cE)$. 
\begin{assumption}[Weight matrix]
\label{as:network}
The nonnegative weight matrix $\bW$, whose $(i, r)$-th component of $\bW$, denoted as $w_{ir}$, is positive if and only if $(i, r) \in \cE$ or $i = r$, is primitive and doubly stochastic, i.e., $\one_n \bW = \one_n$ and $\bW \one_n = \one_n$. 
\end{assumption}

Assumption \ref{as:network} is standard in the decentralized optimization literature \citep{xin2020variance}, and it guarantees that there exists some nonnegative $\lambda$, i.e., spectral gap, such that 
\begin{align*}
    \|\bW - \one_n \one^\top/n\|_2 := \lambda < 1. 
\end{align*}
The assumed weight matrix $\bW$ can be constructed on undirected and connected graphs \citep{olshevsky2014linear}, and also on some directed and strongly connected graphs that are weight-balanced \citep{gharesifard2012distributed}. For instance, the family of directed exponential graphs, is weight-balanced and serves as an important topology configuration in decentralized training \citep{assran2019stochastic}.

%% file: alg.tex
\section{Algorithm Development: \gtnsgdm}  
\label{sec:algrithm} 
We now describe the proposed Algorithm \gtnsgdm\ and discuss the intuition of its construction. We use $\x_i^t$ to denote the estimate of a stationary point for the global cost function $f$ at node $i$ and $t$-th iteration, and recall that $\g_i(\x_i^t, \bxi_i^t)$ denotes the corresponding stochastic gradient returned from a local first-order oracle. Motivated by the error-feedback approach in \citet{yu2023smoothed}, which serves as a momentum-type of variance reduction after applying a nonlinear operator to handle heavy-tailed noise, we also employ local momentum variance reduction 
\begin{align}
\label{eq:vt-update}
    \bv_i^t = \beta \bv_i^{t - 1} + ( 1- \beta) \g_i(\x_i^t, \bxi_i^t),
\end{align}
where $\beta \in [0, 1)$ serves as the momentum coefficient. Then, we use gradient tracking \citep{di2016next} to handle heterogeneous local functions $\{f_i\}_{i=1}^n$. Specifically, we use an estimator $\y_i^t$ to track the global gradient
\begin{align}
\label{eq:gd-tracker}
    \y_i^{t} =  \sum_{r=1}^n w_{ir} \big(\y_r^{t - 1} + \bv_r^t - \bv_r^{t - 1}\big). 
\end{align}
It is known that gradient tracking helps eliminate the dependence on heterogeneity among local functions $\{f_i\}_{i=1}^n$, such as the requirement of bounded gradient similarity. Furthermore, similar to the approach in \citet{liu2025nonconvex}, which uses normalization to address heavy-tailed noise in \textit{centralized} settings, we avoid applying normalization in the recursive updates of the local gradient estimator $\bv_i^t$ in \eqref{eq:vt-update} and the global gradient tracker $\y_i^t$. Instead, normalization is applied only during the update of $\x_i^t$, with step size $\alpha$, and nonnegative mixing weights $\{w_{ir}\}$ where $w_{ir} > 0$ only when $(i, r) \in \cE$ or $i = r$,  
\begin{align}
\label{eq:local}
    \x_i^{t + 1} = \sum_{r=1}^n w_{ir} \Big(\x_r^t - \alpha \frac{\y_r^t}{\| \y_r^t \|} \Big).
\end{align} 
We combine the local updates \eqref{eq:vt-update}\eqref{eq:gd-tracker}\eqref{eq:local} on node $i \in \cV$ and call it \gtnsgdm, Gradient Tracking based Normalized Stochastic Gradient Descent with momentum. When taking $\beta = 0$, this simplifies to momentum-free gradient tracking with normalization in step \ref{eq:local}. However, our analysis shows that $\gtnsgdm$ performs optimally for some $\beta \in (0, 1)$, making \gtnsgdm\ a non-trivial and optimal algorithmic design for the considered problem class. We provide a tabular description for \gtnsgdm\ in Algorithm \ref{alg:gt-pmnsgd}, where all $\{\x_i^0\}$ are initialized from the same point $\ox^0$ for simplicity.

\begin{algorithm}[ht]
\caption{\gtnsgdm\ at each node $i$}
\label{alg:gt-pmnsgd}
\begin{algorithmic}[1] 
    \REQUIRE $\x_i^{-1} = \x_i^0 = \overline{\x}^0, \bv_i^{-1} = \y_i^{-1} = \zero_d, \alpha, \beta, \{w_{ir}\}, T$. \\
    \FOR{\( t = 0 \) to \( T - 1 \)}
        \STATE Sample \( \bxi_i^{t} \); \hfill (random sample for stochastic gradient)
        \STATE \(\bv_i^{t} \leftarrow \beta \bv_i^{t-1} + (1 - \beta ) \g_i(\x_i^{t}, \bxi_i^{t}) \); \hfill (local gradient estimator) 
        \STATE \(\y_i^{t} \leftarrow \sum_{r = 1}^n w_{ir} (\y_r^{t-1} + \bv_r^t - \bv_r^{t-1}) \); \hfill (local gradient tracker)
        \STATE \( \x_i^{t + 1} \leftarrow \sum_{r = 1}^n w_{ir} \big(\x_r^t - \alpha \frac{ \y_r^{t}}{\| \y_r^{t} \| } \big) \); \hfill (peer-to-peer communication)
    \ENDFOR
\end{algorithmic}
\end{algorithm}

\begin{remark}[Why vanilla gradient normalization fails?] Although vanilla normalization is successfully used in \textit{centralized} settings to robustify SGD against heavy-tailed noise \citep{hubler2024gradient}, its direct extension to the \textit{decentralized} settings fails. Suppose we run a vanilla decentralized normalized (noiseless) gradient descent, i.e., in parallel $\forall i \in \cV$,
\begin{align}
\label{eq:vanilla-gd-normalization}
    \x_i^{t + 1} = \sum_{i=1}^n w_{ir} \big(\x_r^t - \alpha \frac{\nabla f_r(\x_r^t)}{\| \nabla f_r(\x_r^t) \|}\big). 
\end{align}
Then, the global average $\ox^t$ would update in the negative direction of \textit{the sum of normalized local gradients}: $ \ox^{t + 1} = \ox^{t} - \frac{\alpha}{n} \sum_{r = 1}^n \frac{\nabla f_r (\x_r^t) }{\| \nabla f_r(\x_r^t)\|}.$ Let for some $t, \forall r \in \cV, \x_r^t = \x_* = \argmin \sum_{i = 1}^n f_i(\x)$, i.e., all nodes hold the optimal global solution that $\sum_{r=1}^n \nabla f_r (\x_*) = \zero$. Since $\|\nabla f_r(\x_*)\|$ can be different quantities for $r = 1, \ldots, n$, due to function heterogeneity, then $\ox^{t+1}$ will move away from $\x_*$. Therefore, vanilla gradient normalization adds some intrinsic errors from heterogeneous local normalizations. By incorporating gradient tracking, we expect that $\y_r^t$ would converge to its global average $\oy^t$, and $\oy^t$ would converge to $(1/n)\sum_{i=r}^n \nabla f_i(\x_r^t)$. In this way, $\x_r^t$ would move along the direction of the \textit{normalized sum of local gradients}, and thus emulating the centralized setting.
\end{remark} 

In the following claim, we further demonstrate that vanilla gradient normalization can cause the iterates $\x_i^t$ to remain arbitrarily far from the optimal solution (see Appendix \ref{sec:proof_claim} for a proof).
\begin{claim}  
\label{cl:normaldsgd}
Consider algorithm \ref{eq:vanilla-gd-normalization}. For any even $n$, for any $B \ge 1$, there exist $\{f_i\}_{i=1}^n$ satisfying Assumptions \ref{as:ldd}-\ref{as:smooth}, a gradient oracle satisfying Assumption \ref{as:noise}, a mixing matrix satisfying Assumption \ref{as:network}, and an initialization $\x_0$, such that the associated parameters $f_*$, $L$, $\sigma$, $\bW$, $\x_0$, are independent of $B$. Then, $\forall T \ge 1, \forall \alpha > 0$, it holds that $\frac{1}{nT}\sum_{t = 0}^{T - 1} \sum_{i=1}^n \E \big[ \| \nabla f(\x_i^t) \| \big] \ge B$.
\end{claim}
We next break the limitations of vanilla gradient normalization in Claim \ref{cl:normaldsgd} by incorporating gradient tracking and momentum variance reduction. This enables the successful use of normalization to suppress heavy-tailed noise while maintaining optimal convergence despite the added nonlinearity.

%% file: theorems.tex
\section{Main Results}
\label{sec:thm}
We present the main convergence results of \gtnsgdm\  and discuss their implications. The detailed analyses are deferred to Appendix \ref{sec:analysis}. We first consider the case where the tail index $p$ is known. 

\begin{theorem}
\label{thm:main}
Let Assumptions \ref{as:ldd}, \ref{as:smooth}, \ref{as:noise}, \ref{as:network} hold. Denote $f(\ox^0) - f_* = \Delta_0, [\nabla f_1(\ox^0), \ldots, \nabla f_n(\ox^0)]^\top \\  = \nabla F(\one_n \otimes \ox^0)$. Take
\begin{align}
\label{eq:thm-alpha}
     \alpha  = \min \Big( 1, \  \sqrt{ \frac{\Delta_0 (1 - \beta)(1- \lambda)}{4L T}}, \sqrt{\frac{\Delta_0(1 - \lambda)}{ 3.5L T}}, \sqrt{\frac{(1 - \lambda)^2 \Delta_0}{2n^{\frac{1}{2}} L T}}  \Big), 
\end{align}
and $1 - \beta = 1/T^{\frac{p}{3p-2}}$. Assume $\beta \ge 1/10$, then the sequence generated by \gtnsgdm\ satisfies that
\begin{align*}
     \frac{1}{nT}\sum_{t = 0}^{T - 1} \sum_{i=1}^n \E \big[ \| \nabla  f(\x_i^t) \|] = O \Big(    \frac{ \sigma  }{n^{1- \frac{1}{p}} T^{\frac{p-1}{3p-2}} } + \frac{1}{T^{\frac{p-1}{3p-2}}} \sqrt{\frac{L \Delta_0}{ 1 - \lambda  } }  +  \frac{\| \nabla f(\ox^0) \|}{ T^{\frac{2p-2}{3p-2}}} +  \sqrt{\frac{3.5L\Delta_0}{(1 - \lambda) T}} \\
     + \sqrt{\frac{n^{\frac{1}{2}} L \Delta_0}{(1 - \lambda)^2 T}}  + \frac{\sigma n^{\frac{1}{2}} }{(1 - \lambda)^{\frac{1}{p}} T^{\frac{p}{3p-2}} }  +   \frac{ \| \nabla F(\one_n \otimes \ox^0)\| }{(1 - \lambda ) n^{\frac{1}{2}} T^{\frac{p}{3p-2}} }  + \frac{ \sigma}{1 - \lambda} \frac{  n^{\frac{1}{2}}}{ T^{\frac{2p-1}{3p-2}} } + \frac{\Delta_0}{ T} \Big). 
\end{align*}
\end{theorem} 

\begin{remark}[Order-optimal rate] Theorem \ref{thm:main} establishes a non-asymptotic upper bound on the mean $\ell_2$ norm stationary gap of \gtnsgdm\ over any finite time horizon $T$. The $O(\cdot)$ here only absorbs universal constants and preserves all problem parameters. It achieves the \textit{optimal} $O\big( 1/T^{\frac{p - 1}{3p-2}} \big)$ convergence rate in terms of $T$ as it matches the lower bound proved in \citet{zhang2020adaptive}. This optimal guarantee is achieved in decentralized settings \textit{for the first time.} 
\end{remark}

\begin{remark}[Speedup in $n$]
\label{rm:n_speedup} We discuss the asymptotic speedup in number of nodes $n$. For sufficiently large $T$ (or sufficiently small target optimality gap), the upper bound in Theorem \ref{thm:main} is dominated by the leading terms $(1/T^{\frac{p-1}{3p-2}}) \big( \sigma / n^{1 - 1/p}  + \sqrt{L \Delta_0 / (1- \lambda)} \big)$. In the high-noise regime $\sigma \gg n^{1-1/p } \sqrt{L \Delta_0 / (1 - \lambda)}$, the upper bound has a speedup factor $n^{1- 1/p }$. In practice, the noise scale (measured by $\sigma$)  in training attention models or in other high-dimensional problems can be very large, and the speedup in $n$ contributes as a noise reduction.
\end{remark} 



When the tail index $p$ is unknown in advance, we establish the following convergence rate. 
\begin{theorem}
\label{thm:p_indep_cvg}
Let Assumptions \ref{as:ldd}, \ref{as:smooth}, \ref{as:noise}, \ref{as:network} hold and take $\alpha$ as in \eqref{eq:thm-alpha}. Take $1- \beta = 1/\sqrt{T}$ and assume $\beta \ge 1/10$. Then \gtnsgdm\ guarantees that
\begin{align*}
    \frac{1}{nT}\sum_{t = 0}^{T - 1} \sum_{i=1}^n \E \big[ \| \nabla  f(\x_i^t) \|] 
    \le    O \Big( \frac{\sigma}{n^{1 - \frac{1}{p}} T^{\frac{p - 1}{2p}}} + \frac{1}{T^{\frac{1}{4}}} \sqrt{\frac{L \Delta_0}{1 - \lambda}} + \frac{\| \nabla f(\ox^0) \|}{\sqrt{T} }  +  \sqrt{\frac{3.5L\Delta_0}{(1 - \lambda) T}}  \\
     + \frac{\sigma n^{\frac{1}{2}}}{(1 - \lambda)^{\frac{1}{p}} \sqrt{T} } + \frac{ \| \nabla F(\one_n \otimes \ox^0)\| }{(1 - \lambda)n^{\frac{1}{2}} \sqrt{T} } + \sqrt{\frac{n^{\frac{1}{2}} L \Delta_0}{(1 - \lambda)^2 T}}  + \frac{ \sigma n^{\frac{1}{2}} }{(1 - \lambda)T^{\frac{2p - 1}{2p}}} + \frac{\Delta_0}{T} \Big).
\end{align*}
\end{theorem}
Theorem \ref{thm:p_indep_cvg} establishes an upper bound of $O(1/T^{\frac{p-1}{2p}})$ when the tail index $p$ is unknown, matching the best-known rate in the \textit{centralized} setting where algorithm parameters do not rely on $p$ \citep{liu2025nonconvex}. While the convergence rate in \citet{yu2023smoothed} is also independent of the knowledge of $p$, it is only for strongly convex functions and its exact rate exponent remains unspecified.

\begin{remark}[Speedup in $n$ and topology independent rate]
\label{rm:improved_dependence}
Consider $p \in (1, 2)$, i.e., the heavy-tailed case with unbounded variance this paper focuses on. When $T$ is sufficiently large (as required to achieve sufficiently small target optimality gap), the upper bound in Theorem \ref{thm:p_indep_cvg} is dominated by $\frac{\sigma}{n^{1-1/p}} \cdot \frac{1}{T^{(p-1)/2p}}$. Significantly, this upper bound is \textit{independent} of network topology $(\lambda)$ and exhibits a speedup factor $n^{1-1/p }$ in all regimes.
\end{remark} 
\begin{remark}[Hyperparameter selection] \reb{ We compare with the centralized work \citet{liu2025nonconvex} in terms of hyperparameter selections. When $p$ is known, the centralized setup requires $\Delta_0, L, \sigma, T$, while our decentralized setup requires $\Delta_0, L, T, \lambda, n$. Notably, our selection is independent of $\sigma$ but requires the network parameters $n, \lambda$. When $p$ is unknown, the centralized setup requires $L, T$, while the decentralized case additionally requires $\Delta_0, \lambda, n$. We note that $n$ is generally easy to estimate, and we give an example where our hyperparameter dependence on $\lambda$ can be eliminated. If all nodes know $n$ and set weights $\bW = \bI_n - (1/n)\bL$ where $\bL$ is the graph Laplacian, one can obtain the estimate $\lambda \le 1 - c/n^3$ \citep{mohar1991laplacian} for some universal constant $c$. Once we have such an upper bound, we can replace the terms involving $1/(1-\lambda)$ with this estimate in equation \eqref{eq:opt_final_bound}; then in the final proof steps, we can select a step size that is independent of $\lambda$ while still obtaining the same order-optimal rate.}
\end{remark} 
\begin{remark}[Dominant terms for small $T$] 
\reb{Both Theorems 1 and 2 order the upper bound terms in an increasing order of the rate exponent of $1/T$, so the dominant terms are clear for sufficiently large $T$. We briefly discuss some cases to illustrate the effects of other parameters including $n$, $\lambda$, and $p$ when $T$ is small. We discuss one varying parameter while assuming other parameters are fixed. First, for network connectivity $1 - \lambda$, $1/{T^{  \frac{p - 1}{3p - 2}}}$ has coefficient $\frac{1}{\sqrt{1 - \lambda}}$ while the faster decaying term $\frac{1}{\sqrt{T}}$ has coefficient $\frac{1}{1 - \lambda}$. When $1 - \lambda \le \frac{n^{1/2}}{T^{p/(3p - 2)}}$, the latter term $\frac{1}{\sqrt{T}} \sqrt{\frac{n^{1/2} L \Delta_0}{(1 - \lambda)^2}}$ will dominate. Second, for $n$, if $n^{1/2} \ge (1 - \lambda)T^{\frac{p}{3p-2}}$, we will also have a dominant term $\frac{1}{\sqrt{T}} \sqrt{\frac{n^{1/2} L \Delta_0}{(1 - \lambda)^2}}$. Third, for $p \in (1, 2]$: When $p$ approaches $1$ from above, $\frac{1}{T^{(p-1)/(3p-2)}}$ will have coefficient $\sqrt{\frac{L \Delta_0}{1 - \lambda}} + \sigma$ and the speedup in terms of the number of nodes will vanish, but the change in $\frac{1}{T^{p/(3p-2)}} \frac{\sigma n^{1/2}}{(1 - \lambda)^{1/p}}$ will not introduce a new dominant term if we focus on the dependence on $1 - \lambda$ since there already exists a slower term $\frac{1}{\sqrt{T}} \frac{n^{1/2}}{1 - \lambda}$. Similarly, when $p$ approaches $2$, the speedup in $n$ is more significant; the term $\frac{1}{T^{p/(3p-2)}} \frac{\sigma n^{1/2}}{(1 - \lambda)^{1/p}}$ only gets smaller and does not change the dominant terms.  Although our rates are established for general $p \in (1, 2]$, we can compare parameter dependence with rates obtained for $p = 2$ only. For example, in \citet{xin2021improved}, for large enough $T$, the upper bound is $$
O\Big( \frac{\Delta_0 + \sigma \sqrt{L}}{n^{1/4} T^{1/4}} + \frac{ \sqrt{n} \sigma L }{ (1 - \lambda^2)^{3/2} \sqrt{T} } + \frac{ L \| \nabla F(\one_n \otimes \ox^0) \| }{ (1 - \lambda^2)^{3/2} T }  \Big). 
$$
If we set $p = 2$ in our bounds, dropping some higher order terms, we have 
$$ 
O\Big( \frac{1}{T^{1/4}} \cdot \big(\sqrt{\frac{L \Delta_0}{ 1 - \lambda  } } + \frac{\sigma}{n^{1/2}}\big) + \frac{1}{\sqrt{T}} \cdot \big( \| \nabla f({\ox}^0) \| + \sqrt{\frac{3.5L\Delta_0}{1 - \lambda}}  + \sqrt{\frac{n^{\frac{1}{2}} L \Delta_0}{(1 - \lambda)^2}} + \frac{\sigma n^{\frac{1}{2}} }{(1 - \lambda)^{ 1/2 } } \big) \Big).
$$
We compare the above two bounds. For coefficients of $1/T^{1/4}$, our bound has dependence $1/n^{1/2}$ which is faster than the $1/n^{1/4}$ in the $p = 2$ case, but the dependence in $1/\sqrt{1 - \lambda}$ is worse. For coefficients of $1/\sqrt{T}$, our bounds have dependence $1/(1 - \lambda)$ while the $p=2$ case has dependence $1/(1-\lambda^2)^{3/2}$, which is strictly larger than ours.}
\end{remark}

%% file: exp.tex
\section{Experiments} 
\label{sec:exp}
We assess the performance of \gtnsgdm\ through numerical experiments. We first conduct studies on synthetic datasets that mimic language modeling under controlled heavy-tailed noise injection, following \citet{lee2025efficient}. We also present experiments on decentralized training of a decoder-only Transformer (GPT) model with 3M parameters on the Multi30k dataset.

\textbf{Baselines.} We compare \gtnsgdm\ with four decentralized baselines: \dsgd\citep{nedic2009distributed}, \gtdsgd\ \citep{xin2020general}, \dsgdgclip\ \citep{sun2024distributed}, and \texttt{SClip-EF} \texttt{-Network} \citep{yu2023smoothed}. \dsgd\ and \gtdsgd\ handle regular stochastic noise with bounded variance. \dsgdgclip\ converges for strongly convex functions under bounded domains or gradients \citep{sun2024distributed}. \sen\ achieves convergence under symmetric noise with bounded $\E \big[ \| \bxi_i^t \|^p \mid \cF_{t- 1} \big]$ for $p = 1$. All methods are initialized identically and tuned via grid search. Detailed baseline descriptions appear in Table \ref{tab:baselines} (Appendix \ref{subsec:baselines}).

\textbf{Graph topology.} We consider three graph topologies: undirected ring, directed exponential, and complete graphs (see \citet{lian2017can, nedic2018network, assran2019stochastic}). Weight matrices use Metropolis weights \citep{xiao2005scheme}.  For synthetic experiments, we set the number of nodes to $n = 20$, we obtain $\lambda = 0.904$, $0.714$, and $0$ for the ring, exponential, and complete graphs, respectively. For Transformer training with $n = 8$, we have corresponding $\lambda = 0.804$, $0.6$, and $0$.

\subsection{Robust linear regression on synthetic tokenized data} 

We use this synthetic experiment to test our convergence rates under controlled heavy-tailed noise. We consider nonconvex regularized linear regression on synthetic data mimicking language tokens. In language modeling, token frequencies exhibit heavy-tailed distributions: few tokens appear frequently, while most are rare but contextually important. We construct the following synthetic dataset $\bX$ of 1k samples of dimension $d = 20$. The first four features simulate frequent tokens and are sampled independently from Bernoulli distributions: the first two from $\mathrm{Bern}(0.9)$ and the next two from $\mathrm{Bern}(0.5)$. The remaining 16 features represent rare tokens, each sampled from $\mathrm{Bern}(0.1)$. The optimal weight $\w_*$ is Gaussian-sampled, with labels $\y = \bX \w_*$. The synthetic dataset $(\bX, \y)$ is evenly distributed over $n = 20$ nodes, where each node $i$ holds a sub-dataset $(\bX_i, \y_i)$, estimate $\w_i$, and a linear regression model with nonconvex robust Tukey's biweight loss function \citep{beaton1974fitting} to estimate $\w_*$. We inject three different \textit{zero-mean} noises, Gaussian noise ($\cN(\zero, 3\bI_d)$), Student's $t$ noise (degrees of freedom $1.5$, scale $1.0$), and Lévy $\alpha$-stable noise (stability parameter $1.5$, skewness parameter $0.5$, scale $1.0$, non-symmetric, multiplied by $0.1$) into the exact gradient, using corrupted stochastic gradients for updates. See Appendix \ref{subsec:synthetic_add} for additional details.

In Figure \ref{fig:ring-synthetic-noises}, we evaluate \gtnsgdm\ against baselines on ring graphs under various gradient noise. \dsgd\ and \gtdsgd\ converge under Gaussian noise but become unstable under heavy-tailed noise. \dsgdgclip\ remains stable but fails to reach the optimum. Both \gtnsgdm\ and \sen\ exhibit robust convergence and near-optimal performance across all scenarios, consistent with their theoretical guarantees under heavy-tailed noise. 


\begin{figure}[ht]
\centering 
\begin{subfigure}[b]{0.32\textwidth}
    \includegraphics[width=\linewidth]{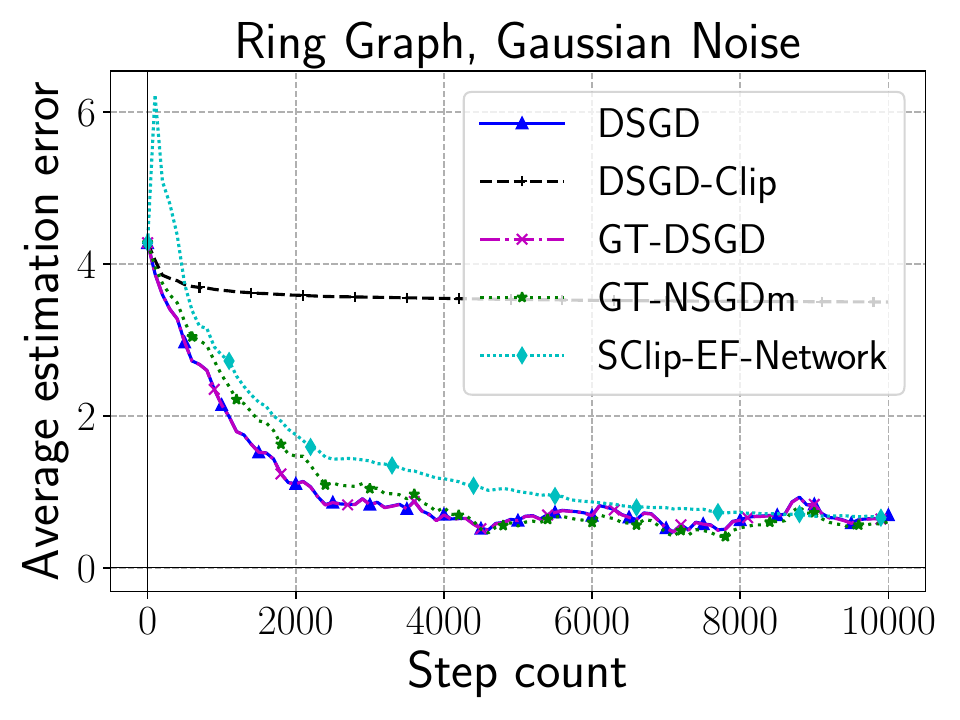}
    \caption{Gaussian Noise}
\end{subfigure}
\hfill
\begin{subfigure}[b]{0.32\textwidth}
    \includegraphics[width=\linewidth]{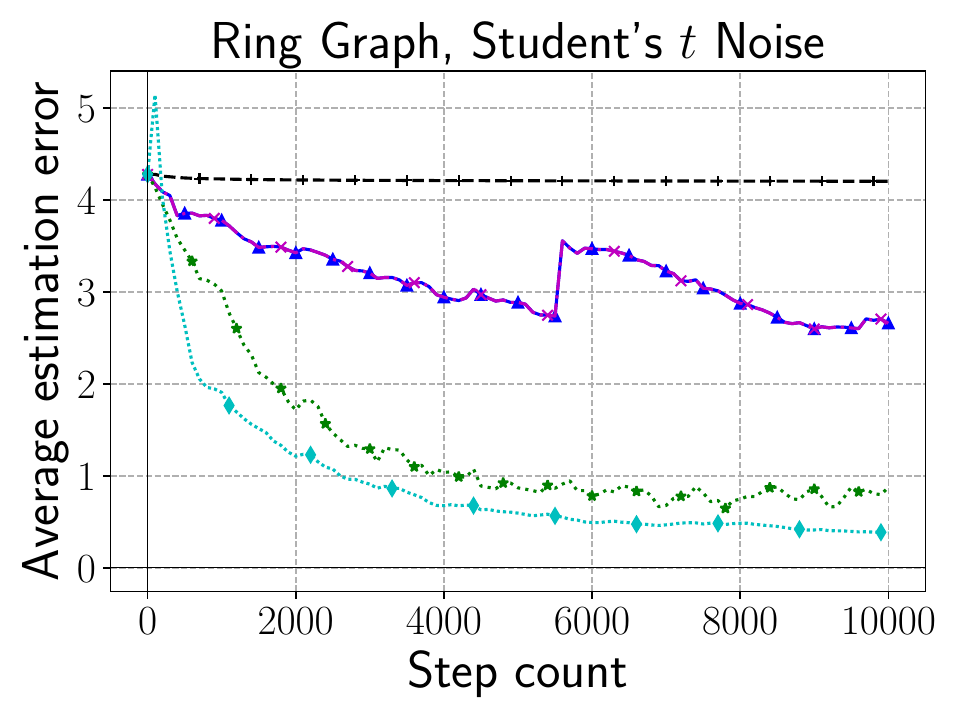}
    \caption{Student's $t$ noise}
\end{subfigure}
\hfill
\begin{subfigure}[b]{0.32\textwidth}
    \includegraphics[width=\linewidth]{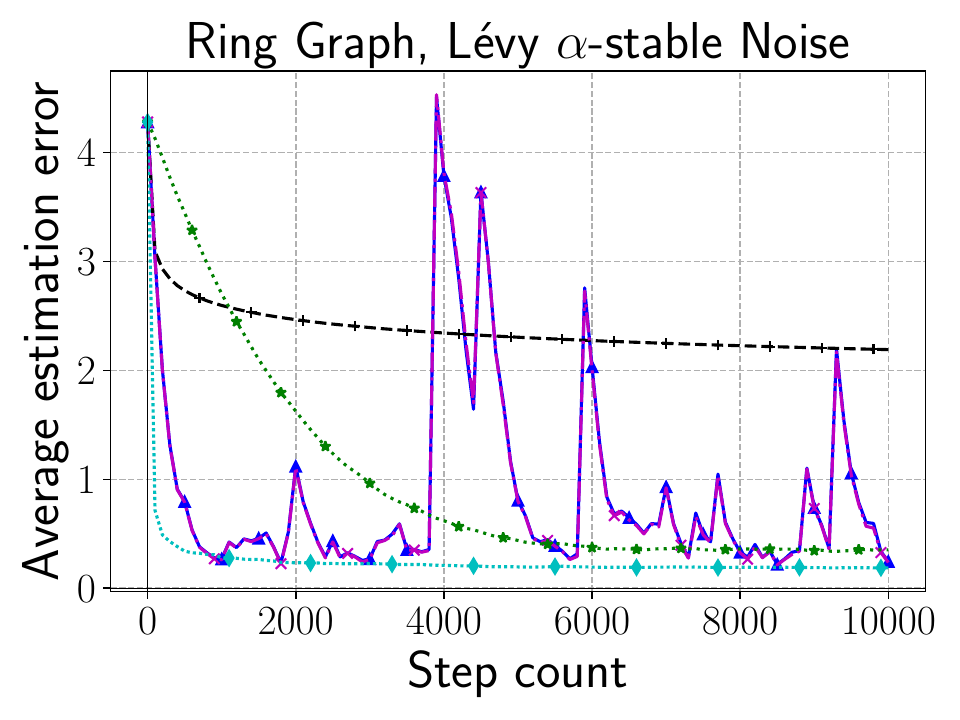}
    \caption{Lévy $\alpha$-stable noise}
\end{subfigure}
\caption{Comparison of performance on a ring graph under various types of injected stochastic gradient noise, measured by the average estimation error $(1/n) \sum_{i=1}^n \| \w_i^t - \w_* \|$ over step count $t$.} 
\label{fig:ring-synthetic-noises}
\end{figure}

In Figure \ref{fig:bounds_dependence}, we test \gtnsgdm’s dependence on connectivity ($\lambda$), noise level ($\sigma$), and the number of nodes ($n$), varying each while fixing others. In Figure \ref{fig:bounds_dependence}(a), we inject Lévy $\alpha$-stable noise and test the performance of \gtnsgdm\ on ring, directed exponential, undirected exponential, and complete graphs with $\lambda = 0.904, 0.714, 0.6, 0$, respectively. \gtnsgdm\ achieves comparable final errors under weak connectivity (i.e., large $\lambda$) versus complete graphs, showing favorable dependence on network connectivity under heavy-tailed noise. In Figure \ref{fig:bounds_dependence}(b), we evaluate \gtnsgdm’s performance under different noise levels on a directed exponential graph. Under Gaussian noise with scale 1 (unit variance), \gtnsgdm\ reaches the best optimality; the final error increases as $\sigma$ grows, as observed under both Gaussian and Lévy $\alpha$-stable noise. In Figure \ref{fig:bounds_dependence}(c), we inject Lévy $\alpha$-stable noise on complete graphs ($\lambda = 0$ for all $n$) with varying number of nodes. As $n$ increases from 2 to 40, convergence speed improves with final errors $[0.4, 0.35, 0.29, 0.20, 0.21]$, demonstrating speedup over certain $n$ ranges, supporting theoretical discussions in Remarks \ref{rm:n_speedup} and \ref{rm:improved_dependence}. 


\begin{figure}[ht!]
\centering 
\begin{subfigure}[b]{0.32\textwidth}
    \includegraphics[width=\linewidth]{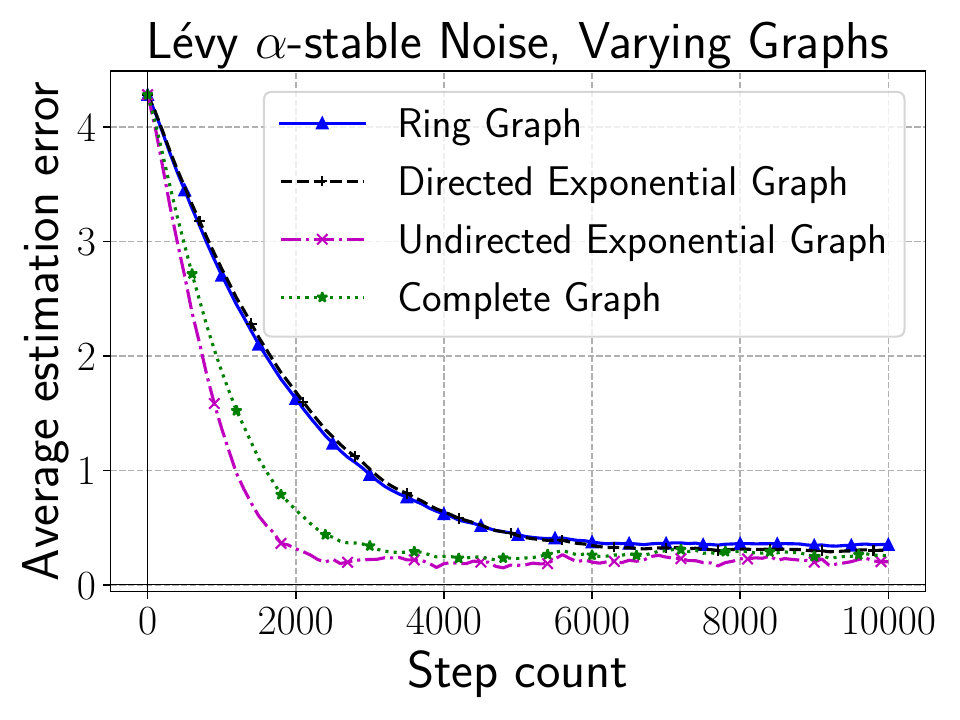}
    \caption{Dependence on $\lambda$}
\end{subfigure}
\hfill
\begin{subfigure}[b]{0.32\textwidth}
    \includegraphics[width=\linewidth]{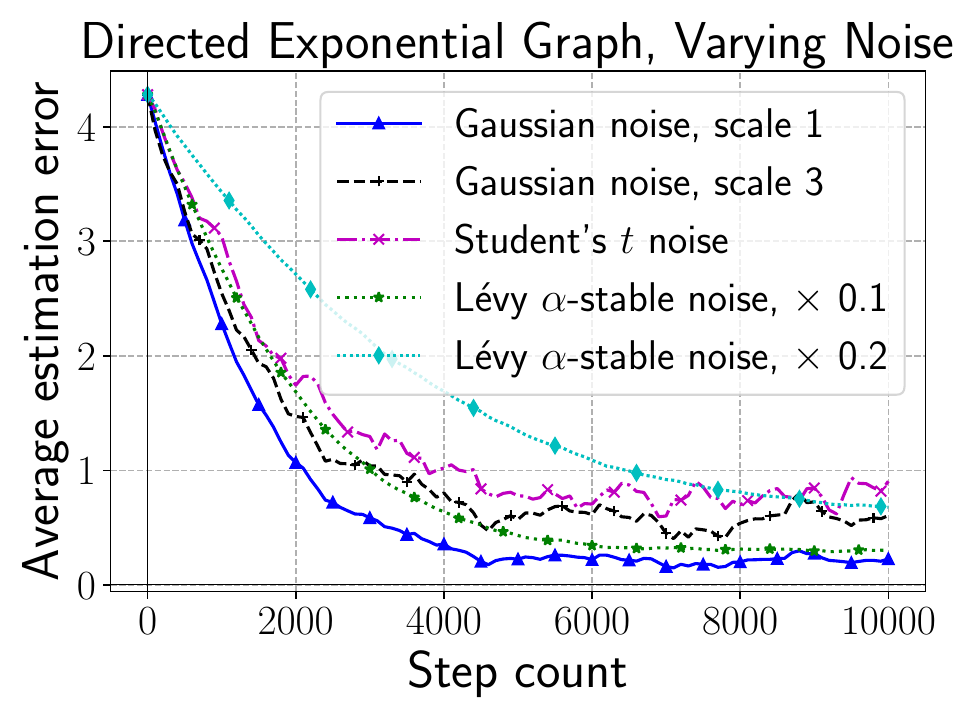}
    \caption{Dependence on $\sigma$}
\end{subfigure}
\hfill
\begin{subfigure}[b]{0.32\textwidth}
    \includegraphics[width=\linewidth]{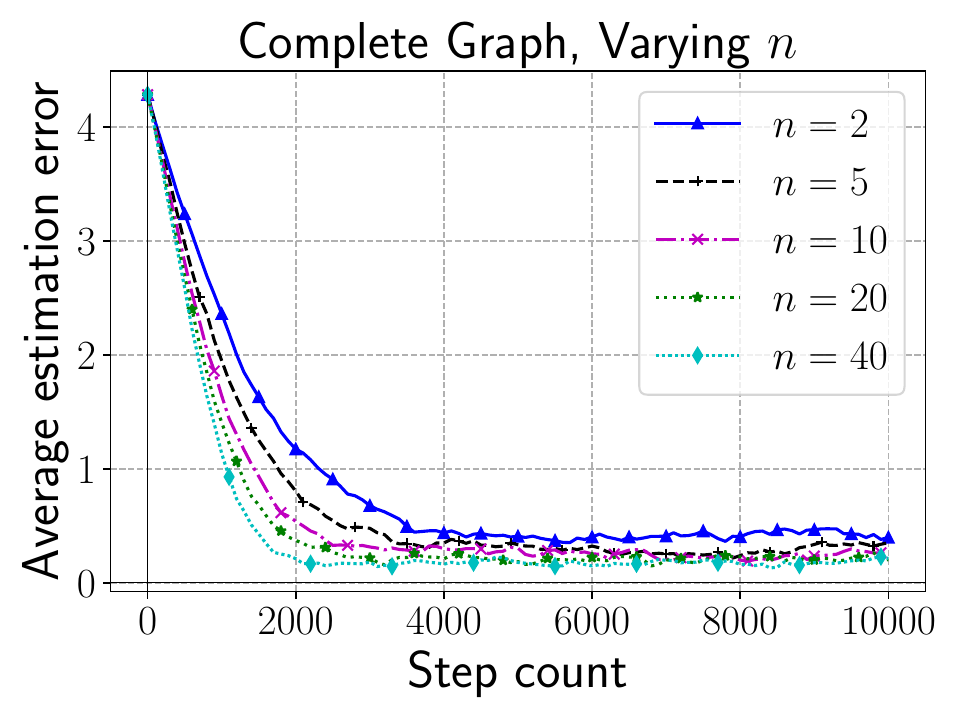}
    \caption{Dependence on $n$}
\end{subfigure}
\caption{Empirical studies on \gtnsgdm's dependence on problem parameters $\lambda, \sigma, n$.}
\label{fig:bounds_dependence}
\end{figure} 

\subsection{Decentralized training of Transformers} 
We evaluate \gtnsgdm's empirical performance on language modeling using a 3M-parameter GPT model \citep{radford2018improving} for auto-regressive modeling on Multi30k (29k sentences, 4.4M tokens). We assess performance using validation log-perplexity. On 8-node graphs with three topologies, we distribute training data evenly and initialize identical GPT models per node. We introduce three additional baselines: \texttt{DSGD-GClip} (\texttt{DSGD} with constant step size and $\ell_2$ gradient clipping level), \texttt{DSGD-CClip} (\texttt{DSGD} with constant step size and component-wise gradient clipping level), \texttt{QG-DSGDm} \citep{lin2021quasi} and \texttt{GT-Adam} \citep{carnevale2022gtadam} (all without theoretical guarantees under heavy-tailed noise; see Table \ref{tab:baselines} in Appendix \ref{subsec:baselines}). We run all methods for 12 epochs with batch size 64. See Appendix \ref{subsec:dec_training} for model and hyperparameter details. 

Table \ref{tab:gpt-training-val} presents average validation loss and standard deviation over five independent runs for each algorithm across three topologies. Results show that \gtnsgdm\ nearly matches the best baseline \texttt{DSGD-GClip} (which lacks theoretical guarantees under heavy-tailed gradient noise) while significantly outperforming the other two theoretically-guaranteed baselines across all topologies. We note that this decentralized training experiment is simulated to demonstrate algorithm effectiveness and has practical limitations.


\begin{table}[h] 
\centering
\renewcommand{\arraystretch}{1.3} 
\caption{Topologies ring, directed exponential (Exp.), and complete (Comp.) graphs. Algorithms are grouped by theoretical (Theo.) guarantees under heavy-tailed noise: with (w/) or without (w/o).} 
\label{tab:gpt-training-val} 
\begin{tabular}{|c|c|c|c|c|}
\hline
Algorithms & Theo. & Ring & Exp. & Comp. \\ \hline
\texttt{DSGD} & w/o & $5.633_{\pm 0.008}$ & $5.632_{\pm 0.007}$ & $5.635_{\pm 0.007}$ \\ \hline
\texttt{DSGD-GClip} & w/o & $\mathbf{0.253}_{\pm 0.007}$ & $\mathbf{0.249}_{\pm 0.010}$ & $\mathbf{0.267}_{\pm 0.010}$ \\ \hline
\texttt{DSGD-CClip} & w/o & $2.725_{\pm 3.179}$ & $5.058_{\pm 2.388}$ & $8.225_{\pm 1.695}$ \\ \hline
\texttt{GT-DSGD} & w/o & $5.362_{\pm 0.002}$ & $5.632_{\pm 0.002}$ & $5.631_{\pm 0.002}$ \\ \hline
\texttt{GT-Adam} & w/o & $0.520_{\pm 0.038}$ & $0.587_{\pm 0.096}$ & $0.524_{\pm 0.045}$ \\ \hline
\texttt{QG-DSGDm} & w/o & $0.394_{\pm 0.007}$ & $0.388_{\pm 0.013}$ & $0.353_{\pm 0.011}$ \\ \hline
\texttt{SClip-EF-Network} & w/ & $5.653_{\pm 0.012}$ & $5.632_{\pm 0.003}$ & $5.636_{\pm 0.004}$ \\ \hline
\texttt{DSGD-Clip} & w/ & $5.633_{\pm 0.004}$ & $5.659_{\pm 0.013}$ & $5.661_{\pm 0.006}$ \\ \hline
\gtnsgdm & w/ & $\mathbf{0.258}_{\pm 0.007}$ & $\mathbf{0.261}_{\pm 0.007}$ & $\mathbf{0.282}_{\pm 0.009}$ \\ \hline
\end{tabular}
\end{table}


%% file: conclusion.tex
\section{Conclusion and Future Work}
\label{sec:conclusion}
In this paper, we have proposed \gtnsgdm\ for solving decentralized nonconvex smooth optimization to address heavy-tailed noise. The key idea is to leverage normalization, together with momentum variance reduction, to combat heavy-tailed noise, and use gradient tracking to handle cross-node heterogeneity and the nonlinearity introduced by normalization. Theoretical analyses establish that \gtnsgdm\ attains an optimal convergence rate when the tail index $p$ is known, and a rate that matches the best centralized one when $p$ is unknown. Extensive experiments on nonconvex linear regression and decentralized Transformer training show that \gtnsgdm\ is robust and efficient under heavy-tailed noise across various topologies, and achieves a speedup in $n$. Future directions include extending the current analysis to other nonlinearities, such as sign and clipping \citep{zhang2020adaptive}, and generalizing \gtnsgdm\ to handle objective functions under relaxed smoothness conditions \citep{liu2025nonconvex}. 

%% file: ack.tex
\section*{Acknowledgments}
We thank the anonymous reviewers for their valuable feedback that improved the presentation of this paper and for identifying minor errors in the proofs. The work of S. Yu and S. Kar is supported in part by NSF grant ECCS 2330196. The work of D. Jakovetic was supported by the Ministry of Science, Technological Development, and Innovation (Grants No. 451-03-33/2026-03/ 200125 \& 451-03-34/2026-03/ 200125); and the Science Fund of Republic of Serbia, project ``LASCADO'' (grant 7359).

%% file: proof_claim.tex
\section{Proof of Claim \ref{cl:normaldsgd}}
\label{sec:proof_claim}
\begin{proof}
Consider $n$ scalar functions that for each $i \in \cV$, $f_i(x) = (1/2)(x - a_i)^2$ for some $a_i$, and complete graph with $\bW = (1/n)\one_n \one_n^\top$. Let $a_i = a, \forall i = 1, \ldots, n/2$, and $a_i = b, \forall i = n/2 + 1, \ldots, n$, and $b - a > 2B + 1$. Let $x_i^0 = a + 0.5, \forall i \in \cV$. Then, $\forall i \in \cV$, vanilla normalization reduces to
\begin{align*}
    x_i^1 & = \frac{1}{n}\sum_{r=1}^n \Big( x_r^0 - \alpha \text{sign}( x_r^0 - a_r) \Big) \\
    & = x_r^0 - \frac{\alpha}{n} \sum_{r=1}^n \text{sign}(x_r^0 - a_r)  \\
    & = x_r^0 - \frac{\alpha}{n} \sum_{r=1}^{n/2} \text{sign}(0.5) - \frac{\alpha}{n} \sum_{r=n/2 + 1}^n \text{sign}(0.5 - (b - a) ) \\
    & = x_r^0.
\end{align*}
Therefore, $x_r^t = a + 0.5, \forall r \in \cV, \forall t \ge 0$. Since the optimal solution to the original problem is $\frac{a + b}{2}$, the optimality gap is $\frac{b - a}{2} - 0.5 \ge B$.
\end{proof}

\begin{remark}
Note that the proof above can be further extended to the case where the gradient oracle admits almost surely bounded gradient noise. We can use the noise bound to adapt the choices of $a, b, \varepsilon$ such that all signs still get canceled. Similar examples have been used  to show divergence results in \citet{shulgin2025smoothed}. 
\end{remark}

%% file: analysis.tex
\section{Proofs of Theorems} 
\textbf{Proof structure}. \reb{The central recursion of our analysis leverages the descent lemma for $L$ smooth functions applied to consecutive network averages of $\{ \boldsymbol{x}_i^{t + 1} \}$ and $\{ \boldsymbol{x}_i^{t} \}$ (Lemma \ref{lm:dclm-main}). This recursion involves two coupled error sources: consensus errors between $\{\boldsymbol{x}_i^{t}\}$, $\{\boldsymbol{y}_i^{t}\}$, and gradient estimation errors. We establish the intricate coupling between these errors through a series of intermediate lemmas. Our proof strategy proceeds as follows: we first derive two key lemmas that bound consensus errors in terms of gradient estimation errors (Lemma \ref{lm:cns} and \ref{lm:cns-y}), then decompose gradient estimation errors into constituent components, including average gradient estimation errors (Lemma \ref{lm:avg-gd-est-err}) and stack gradient estimation errors (Lemma \ref{lm:stacked-gd-est-err}),  and bound each separately. With all error bounds established, we substitute these into the main recursion and optimize hyperparameter selection to achieve the final order-optimal convergence rates.}

\label{sec:analysis}
\subsection{Preliminaries}
We define some stacked long vectors,
\begin{align*}
    F(\x^t) & := [f_1(\x_1^t), \ldots, f_n(\x_n^t)]^\top, \\
    \nabla F(\x^t) & := [\nabla f_1(\x_1^t)^\top, \ldots, \nabla f_n (\x_n^t)^\top ]^\top, \\
    \g(\x^t, \bxi^{t}) & := [\g_1(\x_1^t, \bxi_1^{t})^\top, \ldots, \g_n(\x_n^t, \bxi_n^{t})^\top]^\top \\
    \bv^t & := [(\bv_1^t)^\top, \ldots, (\bv_n^t)^\top]^\top, \\
    \cN(\y^t) & :=  \Big[ \frac{(\y_1^{t})^\top }{\| \y_1^{t} \| }, \ldots, \frac{(\y_n^{t})^\top }{\| \y_n^{t} \| } \Big]^\top, \\ 
    \x^t & := [(\x_1^t)^\top, \ldots, (\x_n^t)^\top]^\top. 
\end{align*}
Then, Algorithm \ref{alg:gt-pmnsgd} can be rewritten in the compact long-vector form:
\begin{align}
\bv^{t} & = \beta  \bv^{t - 1} + (1 - \beta)\g(\x^{t}, \bxi^t); \label{eq:compact-update-v} \\
\y^{t} & = \big( \bW \otimes  \bI_d \big) ( \y^{t - 1} + \bv^t - \bv^{t -1}), \label{eq:compact-update-y}\\
\x^{t + 1} & = \big(  \bW \otimes \bI_d \big) (\x^{t} - \alpha \cN(\y^t)). \label{eq:compact-update}
\end{align}
We define the following averages over network: 
\begin{align}
\label{eq:def-avg}
    \ov^t = \frac{1}{n}\sum_{i=1}^n \bv_i^t, \quad \oy^t = \frac{1}{n} \sum_{i=1}^n \y_i^t, \quad \ty^t = \frac{1}{n}\sum_{i = 1}^n \frac{\y_i^{t}}{\| \y_i^t \|}, \quad \ox^{t} = \frac{1}{n} \sum_{i = 1}^n \x_i^t, \quad \onab F(\x^t) = \frac{1}{n}\sum_{i=1}^n \nabla f_i(\x_i^t). 
\end{align} 
From the doubly stochasticity of $\bW$, the global average updates as 
\begin{align}
\label{chap4-2:eq:avg-update}
\ox^{t + 1} = \ox^t - \frac{\alpha}{n} \sum_{r = 1}^n \frac{\y_r^{t}}{\| \y_r ^ { t} \|} = \ox^{t} - \alpha \ty^{t}. 
\end{align}

\subsection{Intermediate Lemmas}

We first present some standard useful relations to be used in our analysis.
\begin{lemma}
\label{lm:useful}
The following relations hold: 
\begin{enumerate}
    \item $\oy^{t} = \ov^{t} $;
    \item $\bW - \one_n \one_n^\top/n = (\bW - \one_n \one_n^\top/n)(\bI_n - \one_n \one_n^\top/n) = (\bI_n - \one_n \one_n^\top/n) (\bW - \one_n \one_n^\top/n)  $; 
    \item $\bW^k - \one_n \one_n^\top /n = (\bW - \one_n \one_n^\top /n )^k, \forall k \in \N_{+}$;
    \item $ (1/\sqrt{n})\sum_{i=1}^n \| \ba_i \| \le \| \ba \| \le \sum_{i=1}^n \| \ba_i \|$, $\forall \ba = [\ba_1^\top, \ldots, \ba_n^\top]^\top \in \R^{nd}$, 
    \item $ \sum_{i=1}^m a_i^p \le \big( \sum_{i=1}^m a_i\big)^p \le m^{p-1} \sum_{i=1}^m a_i^p, \forall m \in \N_{+}, \forall a_i \in \R_{+}$. 
\end{enumerate}
\end{lemma}

We then present a standard decent lemma for $L$-smooth functions . 
\begin{lemma}[Decent lemma for $L$-smooth functions] 
\label{lm:decent-lemma}
Let Assumption \ref{as:smooth} hold. For any $\x, \y \in \R^d$, there holds
\begin{align*}
    f(\y) \le f(\x) + \nabla f(\x)^\top(\y - \x) + \frac{L}{2}\| \x - \y \|^2. 
\end{align*}
\end{lemma}

We next present the main descent lemma on the network average. 
\begin{lemma}[Decent lemma for network average] 
\label{lm:dclm-main}
Let Assumption \ref{as:smooth} hold. Let $\beps^t = \oy^{t} - \nabla f(\ox^t)$. We have
\begin{align*}
    \sum_{t=0}^{T - 1} \alpha \| \nabla f(\ox^t) \| \le f(\ox^0) - f_* + \sum_{t =0}^{T - 1} 2\alpha\| \beps^t \| + \sum_{t =0}^{T - 1} \frac{\alpha}{n} \sum_{i = 1}^n \| \oy^{t} - \y_i^{t} \|  + \sum_{t =0}^{T - 1} \frac{L}{2} \alpha^2. 
\end{align*}
\end{lemma}
\begin{proof}
Since $\|\ox^{t + 1} - \ox^t\| = \alpha  \| \ty^t \| = \alpha$, 
applying Lemma \ref{lm:decent-lemma} on $\ox^{t + 1}, \ox^t$ gives that
\begin{align} 
    f(\ox^{t + 1}) 
    & \le f(\ox^t) + \nabla f(\ox^t)^\top (\ox^{t + 1} - \ox^t)+ \frac{L}{2}\| \ox^{t + 1} - \ox^t \|^2 
    \notag \\
    &  \overset{(i)}{\le} f(\ox^t) - \alpha (\oy^{t} - \beps^t)^\top \ty^{t} + \frac{L}{2} \alpha^2 \notag  \\
    & \overset{(ii)}{\le} f(\ox^t) - \alpha (\oy^{t})^\top \ty^{t} + \alpha \| \beps^t \| + \frac{L}{2} \alpha^2, \label{eq:avg_dec_1}
\end{align} 
where we used the definitions \eqref{eq:def-avg}\eqref{chap4-2:eq:avg-update} in $(i)$, and used Cauchy-Schwartz inequality followed by $\|\ty^{t}\| \le 1$ in $(ii).$ Next,
\begin{align}
- (\oy^{t})^\top \ty^{t } 
& = - (\oy^{t})^\top \Big[ \frac{\oy^{t}}{\| \oy^{t} \|} + \frac{1}{n}\sum_{i=1}^n \y_i^{t}\big( \frac{1}{\| \y_i^{t} \|} - \frac{1}{\| \oy^{t} \|} \big)  \Big] \notag \\
& \le - \| \oy^{t} \| + \| \frac{1}{n}\sum_{i=1}^n \y_i^{t} \big( \frac{\| \oy^{t}\|}{\| \y_i^{t} \|} - 1\big) \| \notag  \\
& \overset{(i)}{\le} - \|\nabla f(\ox^t) \| + \| \beps^{ t} \| + \frac{1}{n} \sum_{i=1}^n \big| \| \oy^{t} \| - \| \y_i^{t} \| \big| \notag \\
& \overset{(ii)}{\le}  - \|\nabla f(\ox^t) \| + \| \beps^{ t} \| + \frac{1}{n} \sum_{i=1}^n \| \oy^{t} - \y_i^{t} \|, \label{eq:yy}
\end{align} 
where we used $\| \oy^{t} \| = \| \nabla f(\ox^t) + \beps^t \| \ge \| \nabla f(\ox^t) \| - \| \beps^t \| $, and Cauchy-Schwartz inequality in $(i)$, and $| \| \ba \|  - \| \bb \| | \le \| \ba - \bb \|$ for any $\ba, \bb \in \R^d$ in $(ii)$. Plugging in \eqref{eq:yy} into \eqref{eq:avg_dec_1}, and summing over $t = 0, \ldots, T - 1$, we have
\begin{align*}
f(\ox^{T}) \le f(\ox^0) - \sum_{t =0}^{T - 1} \alpha \| \nabla f(\ox^t) \| + \sum_{t =0}^{T - 1} 2\alpha\| \beps^t \| + \sum_{t =0}^{T - 1} \frac{\alpha}{n} \sum_{i = 1}^n \| \oy^{t} - \y_i^{t} \|  + \sum_{t =0}^{T - 1} \frac{L}{2} \alpha^2.
\end{align*}
Using $f(\ox^T) \ge f_*$ and rearranging terms above give the desired result.
\end{proof} 

With Lemma \ref{lm:dclm-main}, it remains to bound the gradient estimation error $\| \beps^t \|$ and the consensus error $\y_i^t - \oy^t$. Let us decompose the gradient estimation error as follows: 
\begin{align}
\label{eq:bpes_decom}
    \beps^t = \oy^{t} - \nabla f(\ox^t) = \ov^{t} - \nabla f(\ox^t) = \underbrace{\ov^t - \onab F(\x^t)}_{:= \beps_1^t \in \R^d } + \underbrace{\onab F(\x^t) - \nabla f(\ox^t)}_{:= \beps_2^t \in \R^d}.
\end{align} 
It is clear that $\beps_1^t$ is the gradient estimation error, and $\beps_2^t$, exploiting the smoothness property in \ref{as:smooth}, can be bounded by the consensus error $\x_i^t - \ox^t$. Since the consensus error is also used in bounding $\beps_1^t$, we need to first bound the consensus errors $\x_i^t - \ox^t$ and $ \y_i^t - \oy^t$.

\begin{lemma}[Consensus errors of $\{\x_i^t\}$]
\label{lm:cns}
We have for all $t = 0, \ldots, T$, 
\begin{align}
    \frac{1}{n} \sum_{i = 1}^n \| \x_i^t - \ox^t \|  & \le \frac{\alpha \lambda}{1 - \lambda}. \label{eq:cns-err} 
\end{align} 
\end{lemma}

\begin{proof}
Using the relation 4 in Lemma \ref{lm:useful} we have
\begin{align}
\label{eq:cns-err-comp-long}
    \frac{1}{n} \sum_{i=1}^n \| \x_i^t - \ox^t \| \le \frac{1}{\sqrt{n}} \| \x^t - \one_n \otimes \ox^t \|. 
\end{align}
From the compact form update in \eqref{eq:compact-update}, we have
\begin{align*}
    \x^t = \big( \bW \otimes \bI_d \big) \x^0 - \alpha \sum_{k = 0}^{t - 1} (\bW \otimes \bI_d)^{t - k}\cN(\y^k). 
\end{align*}
It follows that
\begin{align}
    & \| \x^t - \one_n \otimes \ox^t \| \notag \\
    = \ & \| \big( \bI_{nd} - \frac{1}{n}\one_n\one_n^\top \otimes \bI_d\big) \x^t \| \notag  \\
    \overset{(i)}{=} \ &  \alpha \| \sum_{k=0}^{t - 1} \big( \bI_{nd} - \frac{1}{n}\one_n\one_n^\top \otimes \bI_d  \big) \big( \bW \otimes \bI_d \big)^{t - k} \cN(\y^k)\| \notag \\
    \le  \ & \alpha \| \sum_{k = 0}^{t - 1}  \| \bW^{t - k} - \frac{1}{n}\one_n \one_n^\top \|_2 \| \cN(\y^k) \| \notag \\
    \overset{(ii)}{\le} \ &  \alpha \| \sum_{k = 0}^{t - 1} \| \bW - \frac{1}{n}\one_n \one_n^\top \|_2^{t - k} \| \cN(\y^{k}) \| \notag \\
    \le \ & \alpha \sqrt{n} \sum_{k = 0}^{ t- 1} \lambda^{t - k} \\ 
    \overset{(iii)}{\le} \ &  \frac{\alpha\sqrt{n} \lambda}{1 - \lambda}. \label{eq:cns-err-long}
\end{align}
where we used the double stochasticity of $\bW$ and $\x_i^0 = \ox^0, \forall i \in [n]$ in $(i)$, the relation 3 in Lemma \ref{lm:useful} in $(ii)$, and Assumption \ref{as:network} in $(iii)$. Substituting \eqref{eq:cns-err-long} into \eqref{eq:cns-err-comp-long} gives the desired bound in \eqref{eq:cns-err}. 
\end{proof}

Before proceeding to bound consensus errors for $\{\y_i^t\}$, we present the following bound on vector-valued martingale difference sequence from \citet{liu2025nonconvex}.  
\begin{lemma}
\label{lm:mds}
    Given a sequence of random vectors $\bd_t \in \R^d$, $\forall t$ such that $\E[\bd_t \mid \cF_{t - 1} ] = \zero$ where $\cF_{t} = \sigma(\bd_1, \ldots, \bd_t)$ is the natural filtration, then for any $p \in [1, 2]$, there is 
    \begin{align*}
        \E \Big[ \| \sum_{t = 1}^T \bd_t \| \Big] \le 2 \sqrt{2} \E \Big[ \Big( \sum_{t = 1}^T  \| \bd_t \|^p \Big)^{\frac{1}{p}} \Big], \forall T \ge 0.
    \end{align*}
\end{lemma}

\begin{lemma}[Consensus errors for $\{\y_i^t\}$]
\label{lm:cns-y}
We have for all $t = 0, \ldots, T$, 
\begin{align*}
     & \frac{1}{n}  \E \big[ \sum_{i=1}^n \|\y_i^t - \oy^t\| \big] \\
     \le \ &  2\sqrt{2}  n^{\frac{1}{2}} \big(\frac{1}{\beta}  - 1 \big)  \Big( \sum_{k = 0}^t  \lambda^{(t - k + 1)p}  \Big)^{\frac{1}{p}} \sigma  + \frac{1}{\sqrt{n}}\big( \frac{1}{\beta} - 1 \big) \sum_{k = 0}^t \lambda^{t -k + 1} \E \big[ \| \nabla F(\x^k) - \bv^k \|\big]. 
\end{align*}
\end{lemma}
\begin{proof}
Similar to \eqref{eq:cns-err-comp-long}, we have 
\begin{align}
\label{eq:ycns-comp-long}
     \frac{1}{n} \sum_{i =1}^n \| \y_i^t - \oy^t \| \le \frac{1}{\sqrt{n}} \| \y^t - \one_n \otimes \oy^t \|. 
\end{align} 
Following from \eqref{eq:compact-update-y}, 
\begin{align}
   &  \y^t - \one_n \otimes \oy^t  \\
    \overset{\eqref{eq:compact-update-y}}{=} \ &  \big( \bI_{nd}- \frac{1}{n}\one_n \one_n^\top  \otimes \bI_d \big) \big( \bW \otimes \bI_d \big) (\y^{t - 1} + \bv^t - \bv^{t - 1})  \notag \\
    = \ &  \big( \bW \otimes \bI_d - \frac{1}{n}\one_n \one_n^\top  \otimes \bI_d \big) \y^{t - 1} +  \big( \bW \otimes \bI_d - \frac{1}{n}\one_n \one_n^\top  \otimes \bI_d \big) ( \bv^t - \bv^{t - 1}) \notag \\
    \overset{(i)}{=} \ &   \big( \bW \otimes \bI_d - \frac{1}{n}\one_n \one_n^\top  \otimes \bI_d \big)\big( \bI_{nd} - \frac{1}{n} \one_n \one_n^\top \otimes \bI_d\big) \y^{t - 1} + \big( \bW \otimes \bI_d - \frac{1}{n}\one_n \one_n^\top  \otimes \bI_d \big) ( \bv^t - \bv^{t - 1})  \notag \\
    \overset{(ii)}{=} \ &  \sum_{k = 0}^t \big( \bW \otimes \bI_d - \frac{1}{n}\one_n \one_n^\top  \otimes \bI_d \big)^{t - k + 1 } ( \bv^t - \bv^{t - 1} ), \label{eq:ycns-rcur}
\end{align}
where we used relation 3 in Lemma \ref{lm:useful} in $(i)$ and used $\y_i^0 = \zero_d, \forall i \in [n]$ in $(ii)$. From the update in \eqref{eq:compact-update-v}, we have
\begin{align*}
    \bv^t - \bv^{t - 1} = (\beta - 1) \bv^{t - 1} + (1 - \beta) \g(\x^t, \bxi^t) = (1 - \beta)(\bv^t - \bv^{t- 1}) + (1 - \beta)(\g(\x^t, \bxi^t) - \bv^t).
\end{align*}
Then, there holds, 
\begin{align}
\label{eq:vt-diff}
     \bv^t - \bv^{t - 1} = (\frac{1}{\beta} - 1)(\g(\x^t, \bxi^t) - \bv^t) = (\frac{1}{\beta} - 1)(\g(\x^t, \bxi^t) - \nabla F(\x^t) + \nabla F(\x^t) - \bv^t).  
\end{align}
Putting the relation above into \eqref{eq:ycns-rcur} and applying \eqref{eq:ycns-rcur} recursively, from $\y_i^0 = \oy^0$, we have
\begin{align}
\label{eq:longy-cns}
\begin{split}
    \| \y^t - \one_n \otimes \oy^t \| 
    & \le (\frac{1}{\beta} - 1) \|  \sum_{k = 0}^t \big( \bW \otimes \bI_d - \frac{1}{n}\one_n \one_n^\top  \otimes \bI_d \big)^{t - k + 1 } \big(\g(\x^k, \bxi^k) - \nabla F(\x^k) \big)  \|  \\
    & \quad + (\frac{1}{\beta} - 1) \|  \sum_{k = 0}^t \big( \bW \otimes \bI_d - \frac{1}{n}\one_n \one_n^\top  \otimes \bI_d \big)^{t - k + 1 } \big( \nabla F(\x^k) - \bv^k \big)  \| 
\end{split}
\end{align}
We note that the first half of the right hand side above can be addressed by Lemma \ref{lm:mds}:
\begin{align}
\label{eq:est-error-p1}
\begin{split}
    & \E \Big [ \|  \sum_{k = 0}^t \big( \bW \otimes \bI_d - \frac{1}{n}\one_n \one_n^\top  \otimes \bI_d \big)^{t - k + 1 } \big(\g(\x^k, \bxi^k) - \nabla F(\x^k) \big)  \| \Big] \\
    \le \ & 2 \sqrt{2} \E \Big[ \Big(   \sum_{k = 0}^t  \lambda^{(t- k + 1)p} \| \g(\x^k, \bxi^{k}) - \nabla F(\x^k) \|^p \Big)^{\frac{1}{p}} \Big]. 
\end{split}
\end{align} 
We observe that
\begin{align}
\label{eq:est-error-p1-split}
\begin{split}
    & 2 \sqrt{2} \E \Big[ \Big(   \sum_{k = 0}^t  \lambda^{(t- k + 1)p} \| \g(\x^k, \bxi^{k}) - \nabla F(\x^k) \|^p \Big)^{\frac{1}{p}} \mid \cF_{t - 1} \Big] \\
    \le \ & 2 \sqrt{2} \E \Big[ \Big(   \sum_{k = 0}^t  \lambda^{(t- k + 1)p} \big( \sum_{i=1}^n \| \g_i(\x_i^k, \bxi_i^{k}) - \nabla f_i(\x_i^k) \| \big)^p \Big)^{\frac{1}{p}} \mid \cF_{t - 1} \Big]    \\
    \overset{(i)}{\le} \ & 2 \sqrt{2} \E \Big[ \Big(   \sum_{k = 0}^t\sum_{i=1}^n \lambda^{(t- k + 1)p} n^{p - 1}  \| \g_i(\x_i^k, \bxi_i^{k}) - \nabla f_i(\x_i^k) \|^p \Big)^{\frac{1}{p}} \mid \cF_{t - 1} \Big]   \\
    \overset{(ii)}{\le} \  &  2 \sqrt{2} \Big( \E \big[  \sum_{i=1}^n \lambda^p n^{p - 1} \| \g_i(\x_i^t, \bxi_i^{t}) - \nabla f_i(\x_i^t) \|^p \mid \cF_{t - 1} \big] \\
    & \quad + \sum_{k= 0}^{t - 1} \sum_{i = 1}^n  \lambda^{(t - k + 1)p} n^{p-1} \| \g_i(\x_i^k, \bxi_i^k)  - \nabla f_i(\x^k) \|^p \Big)^{\frac{1}{p}} \\ 
    \overset{(iii)}{\le} \ & 2 \sqrt{2} \Big(  \lambda^p n^p \sigma^p  +  \sum_{k= 0}^{t - 1}\sum_{i = 1}^n  \lambda^{(t - k + 1)p} n^{p-1} \| \g_i(\x_i^k, \bxi_i^k)  - \nabla f_i(\x_i^k) \|^p \Big)^{\frac{1}{p}}, 
\end{split}
\end{align}
where we used relation $5$ from Lemma \ref{lm:useful} in $(i)$, Jensen's inequality in $(ii)$, and Assumption \ref{as:noise} in $(iii)$. From \eqref{eq:est-error-p1}, taking expectations on both sides of \eqref{eq:est-error-p1-split}, and applying the above arguments recursively from $\cF_{t - 2}$ to $\cF_0$, we have
\begin{align*}
    \E \Big [ \|  \sum_{k = 0}^t \big( \bW \otimes \bI_d - \frac{1}{n}\one_n \one_n^\top  \otimes \bI_d \big)^{t - k + 1 } \big(\g(\x^k, \bxi^k) - \nabla F(\x^k) \big)  \| \Big] \le 2\sqrt{2} \Big( \sum_{k = 0}^t \lambda^{(t - k + 1)p}  \Big)^{\frac{1}{p}} n  \sigma.
\end{align*}
Therefore, using the relation above, and \eqref{eq:ycns-comp-long}, \eqref{eq:longy-cns}, we reach the desired relation. 
\end{proof}

We then bound average gradient estimation errors  $\beps_1^t = \ov^t - \onab F(\x^t)$. 

\begin{lemma}[Average gradient estimation errors]
\label{lm:avg-gd-est-err}
 For all $t = 0, \ldots, T$, we have
\begin{align*}
    & \E \big[ \| \ov^t - \onab F(\x^t) \| \big] \le \beta^{t + 1} \|  \nabla f(\ox^0) \|  +  \frac{ 2\sqrt{2}}{n^{1 - \frac{1}{p}}} \Big(  \sum_{k = 0}^t \beta^{(t - k)p} (1- \beta)^p  \Big)^{\frac{1}{p}} \sigma  + \sum_{k=0}^t \beta^{t - k + 1} \Big( \frac{2\alpha \lambda}{1 - \lambda} + \alpha \Big)L.
\end{align*}
\end{lemma}
\begin{proof}
Following from the step 4 in Algorithm \ref{alg:gt-pmnsgd}, $\forall i \in [n]$, 
\begin{align}
\label{eq:avg-gd-est-err-recursion}
\bv_i^t - \nabla f_i(\x_i^t)
& = \beta (\bv_i^{t - 1} - \nabla f_i(\x_i^{t - 1})) + (1 - \beta)( \g_i(\x_i^t, \bxi_i^{t}) - \nabla f_i(\x_i^t) ) +  \beta (\nabla f_i(\x_i^{t -1}) - \nabla f_i(\x_i^{t})). 
\end{align}
Averaging the above relation over $i = 1, \ldots, n$ leads to that: 
\begin{align*}
    \beps_1^t 
    & =  \ov^t - \onab F(\x^t) \\
    & = \beta (\ov^{t - 1} - \onab F(\x^{t - 1})) +  
 (1-\beta)  \cdot \underbrace{ \frac{1}{n} \sum_{i=1}^n ( \g_i(\x_i^t, \bxi_i^{t}) - \nabla f_i(\x_i^t) ) }_{:= \bs^t \in \R^d}  + \beta \cdot \underbrace{\frac{1}{n}\sum_{i=1}^n ( \nabla f_i(\x_i^{t - 1}) - \nabla f_i(\x_i^t) )}_{:= \z^t \in \R^d} \\
 & = \beta^{t + 1} \beps_1^{-1} + \sum_{k=0}^t \beta^{t - k} (1 - \beta) \bs^k + \sum_{k = 0}^t \beta^{t - k + 1} \z^k.
\end{align*}
Taking Euclidean norms on both sides gives that
\begin{align}
\label{eq:epst1_bd}
    \| \beps_1^t \| \le \beta^{t + 1} \| \beps_1^{-1} \| + \| \sum_{k = 0}^{t} \beta^{t - k} (1 - \beta)  \bs^k \| + \| \sum_{k=0}^t \beta^{t - k + 1} \z^k \|. 
\end{align} 
We now bound the terms on the right hand side of \eqref{eq:epst1_bd} one by one. First, 
\begin{align}
\label{eq:espt1-bd-1}
    \| \beps_1^{-1} \| = \| \ov^{-1} -  \frac{1}{n} \sum_{i=1}^n \nabla f_i(\ox^0) \| = \|  \nabla f (\ox^0) \|. 
\end{align}
Second, notice that $\{\beta^{t - k}(1 - \beta)(\g_i(\x_i^k, \bxi_i^{k}) - \nabla f_i(\x_i^k))\}$ is a martingale difference sequence that falls into the pursuit of Lemma \ref{lm:mds}, and thus we obtain
\begin{align}
\label{eq:est-error-p2}
\begin{split}
    & \E \Big[ \| \sum_{k = 0}^{t} \beta^{t - k} (1 - \beta)  \bs^k \|  \Big] \\
    = \ & \frac{1}{n} \E \Big[ \| \sum_{k = 0}^{t} \sum_{i = 1}^n  \beta^{t - k} (1 - \beta)  ( \g_i(\x_i^k,  \bxi_i^{k}) - \nabla f_i(\x_i^k)) \|   \Big] \\
    \le \ & \frac{2\sqrt{2}}{n} \E \Big[  \Big( \sum_{k = 0}^{t} \sum_{i = 1}^n  \| \beta^{t - k} (1 - \beta)  ( \g_i(\x_i^k,  \bxi_i^{k}) - \nabla f_i(\x_i^k)) \|^p \Big)^{\frac{1}{p}}   \Big].
\end{split}
\end{align}
Note that
\begin{align}
\label{eq:est-error-p2-split}
\begin{split}
    & \frac{2\sqrt{2}}{n} \E \Big[  \Big( \sum_{k = 0}^{t} \sum_{i = 1}^n  \| \beta^{t - k} (1 - \beta)  ( \g_i(\x_i^k,  \bxi_i^{k}) - \nabla f_i(\x_i^k)) \|^p \Big)^{\frac{1}{p}} \mid \cF_{t - 1}  \Big] \\
    \overset{(i)}{\le} \ & \frac{2\sqrt{2}}{n}  \Big( \E \Big[   \sum_{k = 0}^{t} \sum_{i = 1}^n   \| \beta^{t - k} (1 - \beta)  (\g_i(\x_i^k,  \bxi_i^{k}) - \nabla f_i(\x_i^k)) \|^p  \mid \cF_{t - 1}  \Big] \Big)^{\frac{1}{p}} \\
    \le \ & \frac{ 2\sqrt{2} }{n} \Big( \E \Big[ \sum_{i = 1}^n (1 - \beta)^p  \| ( \g_i(\x_i^t,  \bxi_i^{t}) - \nabla f_i(\x_i^t)) \|^p  \mid \cF_{t - 1}  \Big] \\
    & \quad + \sum_{k = 0}^{t-1} \sum_{i = 1}^n  \| \beta^{t - k} (1 - \beta)  ( \g_i(\x_i^k,  \bxi_i^{k}) - \nabla f_i(\x_i^k)) \|^p \Big)^{\frac{1}{p}} \\
    \overset{(ii)}{\le} \ & \frac{ 2\sqrt{2}}{n}  \Big( n   (1 - \beta)^p \sigma^p  + \sum_{k = 0}^{t-1} \sum_{i=1}^n \| \beta^{t - k} (1 - \beta)(\g_i(\x_i^k, \bxi_i^k) - \nabla f_i(\x_i^k)) \|^p \Big)^{\frac{1}{p}}, 
\end{split}
\end{align}
where we used Jensen's inequality in $(i)$ and Assumption \ref{as:noise} in $(ii)$. From \eqref{eq:est-error-p2}, taking expectations on \eqref{eq:est-error-p2-split}, and recursively applying the preceding arguments from $\cF_{t - 2}$ to $\cF_0$, we have
\begin{align}
\label{eq:epst1-bd-2}
\begin{split}
    & \E \Big[ \| \sum_{k = 0}^{t} \beta^{t - k} (1 - \beta)  \bs^k \|  \Big] \le   \frac{ 2\sqrt{2}}{n^{1 - \frac{1}{p}}} \Big(  \sum_{k = 0}^t \beta^{(t - k)p} (1- \beta)^p  \Big)^{\frac{1}{p}} \sigma. 
\end{split}
\end{align} 
Third, 
\begin{align}
\label{eq:epst1-bd-3}
\begin{split}
    & \| \sum_{k=0}^t \beta^{t - k + 1} \z^k \| \\
    \le \ & \sum_{k = 0 }^t \beta^{t - k + 1} \| \frac{1}{n} \sum_{i = 1}^n (\nabla f_i(\x_i^{k - 1}) - \nabla f_i(\x_i^{ k}) )\|  \\
    \le \ & \sum_{k = 0 }^t \beta^{t - k + 1} \Big( \frac{1}{n} \sum_{i = 1}^n \| \nabla f_i(\x_i^{k - 1} ) - \nabla f_i(\ox^{k - 1}) \| + \frac{1}{n} \sum_{i = 1}^n \| \nabla f_i(\ox^{k - 1}) - \nabla f_i(\ox^{k} ) \| \\
    & \qquad + \frac{1}{n} \sum_{i = 1}^n \| \nabla f_i (\ox^{k} - \nabla f_i(\x_i^k) \| \Big) \\
    \overset{(i)}{\le} \ & \sum_{k = 0 }^t \beta^{t - k + 1} \Big( L \cdot \frac{1}{n} \sum_{i = 1}^n \| \x_i^{k - 1} - \ox^{ k -1 } \| +  L \cdot \frac{1}{n} \sum_{i = 1}^n  \| \ox^{k - 1} - \ox^k \| + L \cdot  \frac{1}{n} \sum_{i = 1}^n \| \ox^k - \x_i^k \| \Big) \\
    \overset{(ii)}{\le} \ & \sum_{k=0}^t \beta^{t - k + 1} \Big( \frac{2\alpha \lambda}{1 - \lambda} + \alpha \Big)L. 
\end{split}
\end{align}
where in $(i)$ we used Assumption \ref{as:smooth} and in $(ii)$ we used \eqref{eq:cns-err} in Lemma \ref{lm:cns}.  Putting relations \eqref{eq:espt1-bd-1}\eqref{eq:epst1-bd-2}\eqref{eq:epst1-bd-3} together leads to the final bound for this lemma.
\end{proof}

We next bound the stacked gradient estimation errors. 

\begin{lemma}[Stacked gradient estimation errors]
\label{lm:stacked-gd-est-err}
For all $t = 0, \ldots, T$, we have
\begin{align*}
    & \E \big[ \| \bv^t - \nabla F(\x^t) \| \big] \\
    \le \ & \beta^{t + 1} \| \nabla F(\one_n \otimes \ox^0) \| + 2\sqrt{2} \Big(\sum_{k=0}^t \beta^{(t - k)p}(1-\beta)^p \Big)^{\frac{1}{p}} n \sigma  + n \sum_{k=0}^t \beta^{t - k + 1} \Big(  \frac{2\alpha \lambda}{1 - \lambda} + \alpha \Big) L.
\end{align*} 
\end{lemma}
\begin{proof}
Define $\tilde{\beps}_1^t := \bv^t - \nabla F(\x^t) \in \R^{nd}$. Similar to \eqref{eq:avg-gd-est-err-recursion}, we have
\begin{align*}
    \bv^t - \nabla F(\x^t) 
    & = \beta(\bv^{t - 1} - \nabla F(\x^{t - 1} )) + (1 - \beta) \underbrace{(\g(\x^t, \bxi^t)  - \nabla F(\x^t))}_{:= \ts^t \in \R^{nd}} + \beta \underbrace{(\nabla F(\x^{t - 1}) - \nabla F(\x^t))}_{:= \tz^t \in \R^{nd}} \\
    & = \beta^{t + 1}  \tilde{\beps}_1^{-1}  + \sum_{k=0}^t \beta^{t -k} (1 - \beta) \ts^k + \sum_{k=0}^t  \beta^{t - k + 1} \tz^k. 
\end{align*}
Taking Euclidean norms on both sides gives that
\begin{align*}
    \| \tbeps_1^t \| \le \beta^{t + 1} \| \tbeps_1^{-1} \| + \| \sum_{k=0}^t \beta^{t -k} (1 - \beta) \ts^k \| + \| \sum_{k=0}^t  \beta^{t - k + 1} \tz^k \|. 
\end{align*}
Similar to the analysis in  Lemma \ref{lm:avg-gd-est-err}, we bound the right hand side above term by term. First, 
\begin{align*}
    \| \tbeps_1^{-1} \| = \| \nabla F(\one_n \otimes \ox^0) \|. 
\end{align*}
Second, notice also that $\{\beta^{t - k} (1- \beta) \ts^k\}$ is a martingale difference sequence and can be dealt with using Lemma \ref{lm:mds}. We have
\begin{align}
\label{eq:est-error-p3}
\begin{split}
    & \E \big[ \| \sum_{k = 0}^t \beta^{t -k} ( 1- \beta) \ts^k  \|  ] \\
    = \ & \E \big[ \|  \sum_{k = 0}^t \beta^{t - k}(1 - \beta)(\g(\x^t, \bxi^{t, b}) - \nabla F(\x^t) ) \|  \big] \\ 
    \le  \ & 2\sqrt{2} \E\big[ \big( \sum_{k = 0}^t \beta^{(t - k)p} ( 1- \beta)^p \| \g(\x^t, \bxi^{t}) - \nabla F(\x^t) \|^p  \big)^{\frac{1}{p}} \big]. 
\end{split}
\end{align} 
In addition,
\begin{align}
\label{eq:est-error-p3-split}
\begin{split}
    & 2\sqrt{2} \E\big[ \big( \sum_{k = 0}^t \beta^{(t - k)p} ( 1- \beta)^p \| \g(\x^t, \bxi^{t}) - \nabla F(\x^t) \|^p  \big)^{\frac{1}{p}} \mid \cF_{t - 1 } \big] \\
    \overset{(i)}{\le} \ & 2\sqrt{2} \Big( \E \big[ \sum_{k = 0}^t \beta^{(t - k)p} ( 1- \beta)^p \| \g(\x^t, \bxi^{t}) - \nabla F(\x^t) \|^p \big] \mid \cF_{t - 1} \Big)^{\frac{1}{p}} \\
    \overset{(ii)}{\le} \ & 2\sqrt{2} \Big( \E \big[ \sum_{k = 0}^t \beta^{(t - k)p} ( 1- \beta)^p \big( \sum_{i=1}^n \| \g_i(\x_i^t, \bxi_i^{t}) - \nabla f_i(\x_i^t) \| \big)^p \big] \mid \cF_{t - 1} \Big)^{\frac{1}{p}} \\ 
    \overset{(iii)}{\le} \ & 2\sqrt{2} \Big( \E \big[ \sum_{k = 0}^t  \sum_{i=1}^n \beta^{(t - k)p} ( 1- \beta)^p n^{p - 1}  \| \g_i(\x_i^t, \bxi_i^{t}) - \nabla f_i(\x_i^t) \|^p \big] \mid \cF_{t - 1} \Big)^{\frac{1}{p}} \\ 
    = \ & 2\sqrt{2} \Big( \E \big[ \sum_{i=1}^n  (1- \beta)^p n^{p - 1} \| \nabla \g_i(\x_i^{t}, \bxi_i^{t}) - \nabla f_i(\x_i^t)\|^p \mid \cF_{t - 1} \big] \\
    & \quad + \sum_{k = 0}^{t - 1} \sum_{i=1}^n \beta^{(t - k)p} ( 1- \beta)^p n^{p - 1}  \| \g_i(\x_i^t, \bxi_i^{t}) - \nabla f_i(\x_i^t) \|^p \Big)^{\frac{1}{p}}  \\
    \le \ & 2\sqrt{2}\Big( (1- \beta)^p n^p \sigma^p  + \sum_{k = 0}^{t - 1} \sum_{i=1}^n \beta^{(t - k)p} ( 1- \beta)^p n^{p - 1}  \| \g_i(\x_i^t, \bxi_i^{t}) - \nabla f_i(\x_i^t) \|^p \Big)^{\frac{1}{p}},  
\end{split}
\end{align}
where in $(i)$ we used Jensen's inequality, and in $(ii), (iii)$ we used relations $4$,and $5$ in Lemma \ref{lm:useful}, respectively. Based on \eqref{eq:est-error-p3}, taking expectations on both sides of \eqref{eq:est-error-p3-split}, and applying the above arguments from $\cF_{t - 2}$ to $\cF_0$, we obtain
\begin{align*}
    \E \big[ \| \sum_{k = 0}^t \beta^{t -k} ( 1- \beta) \ts^k  \| \big] \le 2\sqrt{2} \Big(\sum_{k=0}^t \beta^{(t - k)p}(1-\beta)^p \Big)^{\frac{1}{p}} n \sigma. 
\end{align*}
Third, 
\begin{align*}
    & \| \sum_{k=0}^t  \beta^{t - k + 1} \tz^k \| \\
    \le \ &  \sum_{k = 0}^t  \beta^{t - k + 1} \| \nabla F(\x^{k-1}) - \nabla F(\x^k) \| \\ 
    \le \ & \sum_{k = 0}^t \sum_{i=1}^n  \beta^{t - k + 1}  \| \nabla f_i(\x_i^{k -1}) -  \nabla f_i(\x_i^k) \| \\
    \le \ & \sum_{k = 0}^t \sum_{i=1}^n  \beta^{t - k + 1}  \Big( \| \nabla f_i(\x_i^{k -1}) - \nabla f_i(\ox^{k-1})\| + \| \nabla f_i(\ox^{k-1}) - \nabla f_i(\ox^k) \| + \| \nabla f_i(\x_i^k) - f_i(\ox^k)\| \Big)\\ 
    \overset{(i)}{\le} \ & n \sum_{k=0}^t \beta^{t - k + 1} \Big(  \frac{2\alpha \lambda}{1 - \lambda} + \alpha \Big) L.
\end{align*}
where $(i)$ follows from similar arguments in \eqref{eq:epst1-bd-2}. 
\end{proof} 

Now we are ready to prove our main theorems.
\begin{proof}[Proof of Theorem \ref{thm:main}]
We observe that 
\begin{align}
\begin{split}
\label{eq:local-measure}
    \frac{1}{n}\sum_{t = 0}^{T - 1} \sum_{i=1}^n \E \big[ \| \nabla  f(\x_i^t) \|] 
    & \le \frac{1}{n}\sum_{t = 0}^{T - 1} \sum_{i=1}^n \E \big[\| \nabla f(\x_i^t) - \nabla f(\ox^t) \| + \| \nabla f(\ox^t) \| \big] \\
    & \overset{\eqref{eq:cns-err}}{\le } T \cdot \frac{\alpha \lambda L}{ 1- \lambda} + \sum_{t = 0}^{T - 1} \E \big[ \| \nabla f(\ox^t) \| \big]. 
\end{split}
\end{align}
From Lemmas \ref{lm:dclm-main}, \eqref{eq:bpes_decom}, and Lemma \ref{lm:cns-y}, 
\begin{align*}
    & \sum_{t=0}^{T - 1} \alpha \| \nabla f(\ox^t) \| \\
    \le \ & f(\ox^0) - f_*  + \sum_{t =0}^{T - 1} 2\alpha \big( \| \beps_1^t \| + \| \onab F(\x^t) - \nabla f(\ox^t)  \| \big)  + \sum_{t =0}^{T - 1} \frac{\alpha}{n} \sum_{i = 1}^n \| \oy^{t} - \y_i^{t} \| + \sum_{t =0}^{T - 1} \frac{L}{2} \alpha^2  \\
    \le \ &  f(\ox^0) - f_* \\ 
    & \quad + \sum_{t = 0}^{T - 1} 2\alpha \Big[ \beta^{t + 1} \|  \nabla f(\ox^0) \|  +  \frac{ 2\sqrt{2}}{n^{1 - \frac{1}{p}}} \Big(  \sum_{k = 0}^t \beta^{(t - k)p} (1- \beta)^p  \Big)^{\frac{1}{p}} \sigma  + \sum_{k=0}^t \beta^{t - k + 1} \Big( \frac{2\alpha \lambda}{1 - \lambda} + \alpha \Big)L + \frac{\alpha \lambda L}{ 1- \lambda}    \Big] \\
    & \quad + \sum_{t = 0}^{T - 1} \alpha \Big[ 2\sqrt{2}  n^{\frac{1}{2}} \big(\frac{1}{\beta}  - 1 \big)  \Big( \sum_{k = 0}^t  \lambda^{(t - k + 1)p}  \Big)^{\frac{1}{p}} \sigma  + \frac{1}{\sqrt{n}}\big( \frac{1}{\beta} - 1 \big) \sum_{k = 0}^t \lambda^{t -k + 1} \E \big[ \| \nabla F(\x^k) - \bv^k \|\big] \Big] \\ 
    & \quad + \frac{1}{2} \alpha^2 L T \\
    \overset{(i)}{\le} \ & f(\ox^0) - f_* + 2 \|\nabla f(\ox^0) \| \cdot \frac{\alpha}{1 - \beta} + \frac{4\sqrt{2}\sigma}{n^{1 - \frac{1}{p}}} \cdot \alpha (1 - \beta)^{1 - \frac{1}{p}} T +  \frac{4L}{ 1- \lambda} \cdot  \frac{\alpha^2 T}{1 - \beta} + \frac{2 L}{ 1- \lambda} \cdot \alpha^2 T \\
    & \quad +   \frac{ 2\sqrt{2} \sigma n^{\frac{1}{2}} }{(1 - \lambda)^{\frac{1}{p}}} \cdot \big(\frac{1}{\beta} - 1 \big) \alpha T + \frac{1}{2} L \cdot \alpha^2 T \\
    & \quad + \frac{1}{\sqrt{n}} \big( \frac{1}{\beta} - 1 \big) \alpha \sum_{t = 0}^{T - 1}\sum_{k=0}^t \lambda^{t - k + 1} \Big( \beta^{t + 1} \| \nabla F(\one_n \otimes \ox^0) \| + 2\sqrt{2}n \sigma (1 - \beta)^{1 - \frac{1}{p}} +  \frac{2n L}{1 - \lambda} \cdot \frac{\alpha \beta}{1 - \beta} \Big)
\end{align*}
where in $(i)$ we used $\beta \le 1, \lambda < 1$ and Lemma \ref{lm:stacked-gd-est-err}. Denote $f(\ox^0) - f_* = \Delta_0$. Dividing $\alpha T$ from the above relation on both sides, and putting it into \eqref{eq:local-measure}, then rearranging terms leads to
\begin{align}
\label{eq:opt_final_bound}
\begin{split}
    & \frac{1}{nT}\sum_{t = 0}^{T - 1} \sum_{i=1}^n \E \big[ \| \nabla  f(\x_i^t) \|]  \\
    \le \ & \frac{\Delta_0}{\alpha T} + \frac{2\| \nabla f(\ox^0) \|}{(1 - \beta) T} + 4\sqrt{2} \sigma \cdot \frac{(1 - \beta)^{1 - \frac{1}{p}}}{n^{1- \frac{1}{p}}} + \frac{4L}{1 - \lambda} \cdot \frac{\alpha}{1 - \beta} + \Big( \frac{3L}{ 1- \lambda} + \frac{L}{2} \Big)\alpha \\
    & \quad + \frac{2\sqrt{2}\sigma}{(1 - \lambda)^{\frac{1}{p}}} \cdot n^{\frac{1}{2}}\big( \frac{1}{\beta} - 1\big) +   \frac{ \| \nabla F(\one_n \otimes \ox^0)\| }{1 - \lambda} \cdot \frac{\frac{1}{\beta} - 1}{n^{\frac{1}{2}}} + \frac{2\sqrt{2} \sigma}{1 - \lambda} \cdot n^{\frac{1}{2}} \big( \frac{1}{\beta} - 1 \big) (1 - \beta)^{1 - \frac{1}{p}} \\
    & \quad + \frac{2  L}{( 1 - \lambda)^2 } \cdot n^{\frac{1}{2}} \alpha \\
    \overset{(i)}{\le} \ & \frac{\Delta_0}{\alpha T} + \frac{2\| \nabla f(\ox^0) \|}{(1 - \beta) T} + 4\sqrt{2} \sigma \cdot \frac{(1 - \beta)^{1 - \frac{1}{p}}}{n^{1- \frac{1}{p}}} + \frac{4L}{1 - \lambda} \cdot \frac{\alpha}{1 - \beta} + \frac{3.5L}{ 1- \lambda} \alpha \\
    & \quad + \frac{20\sqrt{2}\sigma}{(1 - \lambda)^{\frac{1}{p}}} \cdot n^{\frac{1}{2}}(1 - \beta) +   \frac{ 10 \| \nabla F(\one_n \otimes \ox^0)\| }{1 - \lambda} \cdot \frac{1 - \beta}{n^{\frac{1}{2}}} + \frac{20 \sqrt{2} \sigma}{1 - \lambda} \cdot n^{\frac{1}{2}}  (1 - \beta)^{2 - \frac{1}{p}}  + \frac{2  L}{( 1 - \lambda)^2 } \cdot n^{\frac{1}{2}} \alpha \\ 
    \overset{(ii)}{\le} \ &   O \Big( \frac{\Delta_0}{ T} + \frac{2\| \nabla f(\ox^0) \|}{(1 - \beta) T} + 4\sqrt{2} \sigma \cdot \frac{(1 - \beta)^{1 - \frac{1}{p}}}{n^{1- \frac{1}{p}}} + \sqrt{\frac{L \Delta_0}{(1 - \lambda)(1- \beta) T}} +  \sqrt{\frac{3.5L\Delta_0}{(1 - \lambda) T}} \\
    & \quad + \frac{20\sqrt{2}\sigma}{(1 - \lambda)^{\frac{1}{p}}} \cdot n^{\frac{1}{2}}(1 - \beta) +   \frac{ 10 \| \nabla F(\one_n \otimes \ox^0)\| }{1 - \lambda} \cdot \frac{1 - \beta}{n^{\frac{1}{2}}} + \frac{ \sigma}{1 - \lambda} \cdot n^{\frac{1}{2}}  (1 - \beta)^{2 - \frac{1}{p}}  + \sqrt{\frac{n^{\frac{1}{2}} L \Delta_0}{(1 - \lambda)^2 T}} \Big) \\
    \overset{(iii)}{\le} \ &   O \Big( \frac{\Delta_0}{ T} + \frac{\| \nabla f(\ox^0) \|}{ T^{\frac{2p-2}{3p-2}}} +    \frac{ \sigma  }{n^{1- \frac{1}{p}} T^{\frac{p-1}{3p-2}} } + \sqrt{\frac{L \Delta_0}{(1 - \lambda) T^{\frac{2p-2}{3p-2}}}} +  \sqrt{\frac{3.5L\Delta_0}{(1 - \lambda) T}} \\
    & \quad + \frac{\sigma n^{\frac{1}{2}} }{(1 - \lambda)^{\frac{1}{p}} T^{\frac{p}{3p-2}} }  +   \frac{ \| \nabla F(\one_n \otimes \ox^0)\| }{(1 - \lambda ) n^{\frac{1}{2}} T^{\frac{p}{3p-2}} }  + \frac{ \sigma}{1 - \lambda} \frac{  n^{\frac{1}{2}}}{ T^{\frac{2p-1}{3p-2}} }  + \sqrt{\frac{n^{\frac{1}{2}} L \Delta_0}{(1 - \lambda)^2 T}} \Big) \\
\end{split}
\end{align} 
where in $(i)$ we take $\beta \ge 1/10$, in $(ii)$ we used 
\begin{align}
\label{eq:alpha-choice}
    \alpha  = \min \Big( 1, \  \sqrt{ \frac{\Delta_0 (1 - \beta)(1 - \lambda)}{4L T}}, \sqrt{\frac{\Delta_0(1 - \lambda)}{3.5LT}}, \sqrt{\frac{(1 - \lambda)^2 \Delta_0}{2n^{\frac{1}{2}} L T}}  \Big), 
\end{align}
and in $(iii)$ we used $1 - \beta = \frac{1}{T^{\frac{p}{3p-2}}}$. 
\end{proof} 

\begin{proof}[Proof of Theorem \ref{thm:p_indep_cvg}] 
Note that \eqref{eq:opt_final_bound}(ii) still holds under the same choice of $\alpha$ in \eqref{eq:alpha-choice} and $\beta \ge 1/10$. Continuing with $1 - \beta = 1/\sqrt{T}$, we have
\begin{align*}
    & \frac{1}{nT}\sum_{t = 0}^{T - 1} \sum_{i=1}^n \E \big[ \| \nabla  f(\x_i^t) \|] \\
    \le \ &   O \Big( \frac{\Delta_0}{ T} + \frac{\| \nabla f(\ox^0) \|}{\sqrt{T} } +  \frac{\sigma}{n^{1 - \frac{1}{p}}} \cdot \frac{1}{T^{\frac{p - 1}{2p}}} + \frac{1}{T^{\frac{1}{4}}} \sqrt{\frac{L \Delta_0}{1 - \lambda}} +  \sqrt{\frac{3.5L\Delta_0}{(1 - \lambda) T}} + \frac{\sigma n^{\frac{1}{2}}}{(1 - \lambda)^{\frac{1}{p}}} \frac{1}{\sqrt{T}} +  \\
    & \frac{ \| \nabla F(\one_n \otimes \ox^0)\| }{(1 - \lambda)n^{\frac{1}{2}}} \cdot \frac{1}{ \sqrt{T} } + \frac{ \sigma n^{\frac{1}{2}} }{1 - \lambda} \cdot  \frac{1}{T^{\frac{2p - 1}{2p}}}  + \sqrt{\frac{n^{\frac{1}{2}} L \Delta_0}{(1 - \lambda)^2 T}} \Big).
\end{align*}
Rearranging above terms leads to the desired upper bound.
\end{proof} 

%% file: exp_add.tex
\section{Additional Experiment Details} 
\label{sec:app_exp}
\subsection{Baseline descriptions}
\label{subsec:baselines}

Please see Table \ref{tab:baselines} for detailed descriptions of baselines.

\newcolumntype{L}[1]{>{\raggedright\arraybackslash}p{#1}} 
\renewcommand{\arraystretch}{1.5} 
\begin{table}[ht!]
\centering
\caption{Summary of Baseline Methods}
\label{tab:baselines}
\begin{tabular}{@{} L{2.6cm} L{7.6cm} L{2.8cm} @{}}
\toprule
\textbf{Method} & \textbf{Parallel update on node $i$} & \textbf{Hyper-parameters}\\
\midrule
\dsgd & $\x_i^{t + 1} = \sum_{r=1}^n w_{ir} \big(\x_r^{t} - \alpha g_r(\x_r^t, \bxi_r^t) \big)$ &  $\alpha$: constant stepsize  \\
\hline 
\texttt{DSGD-GClip} & $\x_i^{t + 1} = \sum_{r=1}^n w_{ir} \x_r^t - \alpha \clip(g_i(\x_i^t, \bxi_i^t), \tau)$ & $\alpha, \tau$: stepsize $\alpha$, and $\ell_2$ clipping levels $\tau$   \\
\hline 
\texttt{DSGD-CClip} & $\x_i^{t + 1} = \sum_{r=1}^n w_{ir} \x_r^t - \alpha \clip(g_i(\x_i^t, \bxi_i^t), \tau)$ & $\alpha, \tau$: stepsize $\alpha$, and component-wise clipping levels $\tau$   \\
\hline 
\dsgdgclip & $\x_i^{t + 1} = \sum_{r=1}^n w_{ir} \x_r^t - \alpha_t \clip(g_i(\x_i^t, \bxi_i^t), \tau_t)$ & $\alpha, \tau$: stepsize $\alpha_t = \alpha / ( t+ 1)$, and $\ell_2$ clipping levels $\tau_t = \tau(t+1)^{2/5}$   \\
\hline
\gtdsgd & $\begin{array}{l} \y_i^{t + 1} = \sum_{r = 1}^n w_{ir} \big( \y_r^t + g_r(\x_r^t, \bxi_r^t) - g_r(\x_r^{t - 1}, \bxi_r^{t - 1}) \big) \\ \x_i^{t + 1} = \sum_{r=1}^n w_{ir} \big( \x_i^t - \alpha \y_r^{t + 1} \big)\end{array}$ & $\alpha$: constant stepsize  \\
\hline
\texttt{GT-Adam} & $\begin{array}{l} \m_i^{t + 1} = \beta_1 \m_i^t + ( 1 - \beta_1) \bs_i^t \\
\bv_i^{t + 1} = \min\big( \beta_2 \bv_i^t + (1 - \beta_2) \bs_i^t \odot \bs_i^t, G\big) \\
\x_i^{t + 1} = \sum_{r = 1}^n w_{ir} \x_r^t - \alpha \frac{\m_i^{t + 1}}{\sqrt{\bv_i^{ t+ 1} + \epsilon}} \\
\g_i^{t + 1} = \nabla f_i(\x_i^{t + 1}) \\
\bs_i^{t + 1} = \sum_{r=1}^n w_{ir} \bs_r^t + \g_i^{t + 1} - \g_i^t \end{array}$ & $\alpha, G$: constant stepsize $\alpha$, and upper bound $G$, stabilization factor $\epsilon$ \\
\hline
\texttt{QG-DSGDm} & 
$ \begin{array}{l} \m_i^{t + 1} = \beta \hat{\m}_i^t + g_i(\x_i^t, \bxi_i^t) \\
\x_i^{t + 1} = \sum_{r = 1}^n w_{ir} \big( \x_i^t - \eta \m_i^{t + 1} \big) \\
\bd_i^t = (\x_i^{t + 1} - \x_i^t) / \eta \\ 
\hat{\m}_i^{t + 1} = \mu \hat{\m}_i^t + (1 - \mu) \bd_i^t \end{array}$ &  $\eta, \beta, \mu$: constant stepsize $\eta$, momentum parametes $\beta, \mu$ \\
\hline 
\sen & $\begin{array}{l} \m_i^{t + 1} =  \beta_t \m_i^t + (1 - \beta_t) \bPsi_t(g_i(\x_i^t, \bxi_i^t) - \m_i^t) \\
\x_i^{t + 1} = \sum_{r=1}^n w_{ir} \big( \x_r^t - \alpha_t \m_r^{t + 1} \big)
\end{array}$ & $c_\varphi, \tau, \alpha, \beta$: Component-wise smooth clipping operator: $\bPsi_t(y) = \frac{c_\varphi}{\sqrt{t + 1}} \frac{y }{\sqrt{y^2 + \tau(t+1)^{3/5}}}$, stepsize $\alpha_t = \alpha/(t+1)^{1/5}$, momentum stepsize $\beta_t = \beta/\sqrt{t + 1}$. \\ 
\bottomrule
\end{tabular}
\end{table} 

\subsection{Additional details for synthetic experiments}
\label{subsec:synthetic_add}
\textbf{Loss function}. Let $(\bX_{i, k}, \y_{i, k})$ denote the $k$-th sample of sub-dataset $(\bX_i, \y_i)$ on node $i$. The loss function of the considered nonconvex linear regression model on this sample is $\ell(\y_{i, k} -  \bX_{i, k} \w_i^t)$, where the 
\begin{align*}
    \ell(r)= \begin{cases}\frac{c^2}{6}\left(1-\left[1-\left(\frac{r}{c}\right)^2\right]^3\right) & \text { if }|r| \leq c, \\ \frac{c^2}{6} & \text { otherwise }\end{cases}, 
\end{align*}
and we use the suggested value $c = 4.6851$ in the robust statistics literature. 

\textbf{Hyperparameter tuning}. 
Please see Table \ref{tab:synthetic_hyperparameters} for hyperparameter searching ranges for this experiment.

\newcolumntype{L}[1]{>{\raggedright\arraybackslash}p{#1}} 
\renewcommand{\arraystretch}{1.5} 
\begin{table}[ht!]
\centering
\caption{Hyperparameter grid search in synthetic experiments}
\label{tab:synthetic_hyperparameters}
\begin{tabular}{@{} L{2.8cm} L{7.6cm} L{2.8cm} @{}}
\toprule
\textbf{Method} & \textbf{Hyperparameter search set} \\
\midrule
\dsgd & $\alpha \in \{ 10^{-5}, 5*10^{-5}, 10^{-4}, 5*10^{-4}, 10^{-3}, 5*10^{-3}, 10^{-2}, 5*10^{-2}, 10^{-1}, 0.5, 1, 5, 10 \}$  \\
\hline
\dsgdgclip & $\alpha \in \{ 10^{-5}, 5*10^{-5}, 10^{-4}, 5*10^{-4}, 10^{-3}, 5*10^{-3}, 10^{-2}, 5*10^{-2}, 10^{-1}, 0.5, 1, 5, 10 \}, \tau \in \{ 10^{-3}, 5*10^{-3}, 10^{-2}, 5*10^{-2}, 10^{-1}, 0.5, 1, 5, 10, 50, 10^{2} \}$  \\
\hline
\gtdsgd & $\alpha \in \{ 10^{-5}, 5*10^{-5}, 10^{-4}, 5*10^{-4}, 10^{-3}, 5*10^{-3}, 10^{-2}, 5*10^{-2}, 10^{-1}, 0.5, 1, 5, 10 \}$  \\
\hline
\gtnsgdm & $\alpha \in \{10^{-5}, 5*10^{-5}, 10^{-4}, 5*10^{-4}, 10^{-3}, 5*10^{-3}, 10^{-2}, 5*10^{-2}, 10^{-1}, 0.5, 1, 5, 10 \}, \beta \in \{0.01, 0.1, 0.2, 0.3, 0.4, 0.5, 0.6, 0.7, 0.8, 0.9, 0.99\}$  \\
\hline
\sen & $\alpha \in \{ 10^{-3}, 10^{-2}, 0.1, 1, 10, 30 \}, \beta \in \{ 10^{-2}, 0.1, 0.5, 0.8, 0.99 \}, c_\varphi \in \{ 1, 5, 10, 20, 30, 50 \}, \tau \in \{0.1, 1, 10, 50, 100\}$ \\ 
\bottomrule
\end{tabular}
\end{table} 

\textbf{Hardware}. We ran this experiment on Mac OS X 15.3, CPU M4 10 Cores, RAM 16GB. 

\subsection{Additional details for decentralized training of Transformers}
\label{subsec:dec_training}

\textbf{Transformer architecture}. We consider the following decoder-only Transformer model (GPT): vocabulary size is 10208, context length is 64, embedding size is 128, number of attention heads is 4, number of attention layers is 2, the linear projection dimension within attention block is 512, and LayerNorm is applied after the 2nd attention block. The total number of parameters of this model is 3018240. 



\textbf{Hyperparameter tuning.} See Table \ref{tab:minigpt_hyperparameters} for our grid search range for algorithm hyperparameters.

\newcolumntype{L}[1]{>{\raggedright\arraybackslash}p{#1}} 
\renewcommand{\arraystretch}{1.5} 
\begin{table}[ht!]
\centering
\caption{Hyperparameter grid search in decentralized training of Transformers}
\label{tab:minigpt_hyperparameters}
\begin{tabular}{@{} L{2.8cm} L{7.6cm} L{2.8cm} @{}}
\toprule
\textbf{Method} & \textbf{Hyperparameter search set} \\
\midrule
\dsgd & $\alpha \in \{ 10^{-4}, 5*10^{-4}, 10^{-3}, 5*10^{-3}, 10^{-2}, 5*10^{-2}, 10^{-1}, 0.5, 1 \}$  \\
\hline
\texttt{DSGD-GClip} & $\alpha \in \{ 10^{-4}, 10^{-3}, 10^{-2}, 10^{-1}, 1, 10, 10^2 \}, \tau \in \{ 10^{-3}, 10^{-2}, 10^{-1}, 1, 10, 10^{2} \}$  \\
\hline
\texttt{DSGD-CClip} & $\alpha \in \{ 10^{-4}, 10^{-3}, 10^{-2}, 10^{-1}, 1, 10, 10^2 \}, \tau \in \{ 10^{-3}, 10^{-2}, 10^{-1}, 1, 10, 10^{2} \}$  \\
\hline
\dsgdgclip & $\alpha \in \{ 10^{-4}, 10^{-3}, 10^{-2}, 10^{-1}, 1, 10, 10^2 \}, \tau \in \{ 10^{-3}, 10^{-2}, 10^{-1}, 1, 10, 10^{2} \}$  \\
\hline
\gtdsgd & $\alpha \in \{ 10^{-4}, 5*10^{-4}, 10^{-3}, 5*10^{-3}, 10^{-2}, 5*10^{-2}, 10^{-1}, 0.5, 1 \}$  \\
\hline
\texttt{GT-Adam} & $\alpha \in \{ 5*10^{-5}, 10^{-4}, 5*10^{-4}, 10^{-3}, 5*10^{-3}, 10^{-2}, 5*10^{-2}, 10^{-1}, 0.5, 1, 5, 10 \}, G \in \{ 10^{-3}, 10^{-2}, 10^{-1}, 1, 10\}, \epsilon = 10^{-8}$  \\
\hline
\texttt{QG-DSGDm} & $\eta \in \{ 5*10^{-5}, 10^{-4}, 5*10^{-4}, 10^{-3}, 5*10^{-3}, 10^{-2}, 5*10^{-2}, 10^{-1}, 0.5, 1, 5, 10 \}, \beta = \mu \in \{ 0.01, 0.2, 0.4, 0.6, 0.8, 0.99 \}$  \\
\hline
\gtnsgdm & $\alpha \in \{ 10^{-4},  10^{-3},  10^{-2}, 10^{-1}, 1, 10 \}, \beta \in \{ 0.01, 0.2, 0.4, 0.6, 0.8, 0.99\}$  \\
\hline
\sen & $\alpha \in \{ 10^{-4}, 10^{-3}, 10^{-2}, 10^{-1}, 10^{0}, 10^{1}, 10^{2} \}, \beta \in \{ 0.01, 0.4, 0.8, 0.99 \}, c_\varphi \in \{ 0.1, 1, 10, 10^2 \}, \tau \in \{0.01, 0.1, 1, 10\}$ \\ 
\bottomrule
\end{tabular}
\end{table} 

\textbf{Hardware}. We simulate the distributed training on one NVIDIA H100 GPU, using PyTorch 3.2 with CUDA 12. The total hyperparameter search and training procedure took around 100 GPU hours.

%% file: main.bbl
\begin{thebibliography}{10}

\bibitem{tsitsiklis1986distributed}
J.~Tsitsiklis, D.~Bertsekas, and M.~Athans, ``Distributed asynchronous
  deterministic and stochastic gradient optimization algorithms,'' {\em IEEE
  transactions on automatic control}, vol.~31, no.~9, pp.~803--812, 1986.

\bibitem{nedic2018network}
A.~Nedi{\'c}, A.~Olshevsky, and M.~G. Rabbat, ``Network topology and
  communication-computation tradeoffs in decentralized optimization,'' {\em
  Proceedings of the IEEE}, vol.~106, no.~5, pp.~953--976, 2018.

\bibitem{li2020federatedspm}
T.~Li, A.~K. Sahu, A.~Talwalkar, and V.~Smith, ``Federated learning:
  Challenges, methods, and future directions,'' {\em IEEE signal processing
  magazine}, vol.~37, no.~3, pp.~50--60, 2020.

\bibitem{kairouz2021advances}
P.~Kairouz, H.~B. McMahan, B.~Avent, A.~Bellet, M.~Bennis, A.~N. Bhagoji,
  K.~Bonawitz, Z.~Charles, G.~Cormode, R.~Cummings, {\em et~al.}, ``Advances
  and open problems in federated learning,'' {\em Foundations and
  Trends{\textregistered} in Machine Learning}, vol.~14, no.~1--2, pp.~1--210,
  2021.

\bibitem{brisimi2018federated}
T.~S. Brisimi, R.~Chen, T.~Mela, A.~Olshevsky, I.~C. Paschalidis, and W.~Shi,
  ``Federated learning of predictive models from federated electronic health
  records,'' {\em International journal of medical informatics}, vol.~112,
  pp.~59--67, 2018.

\bibitem{lian2017can}
X.~Lian, C.~Zhang, H.~Zhang, C.-J. Hsieh, W.~Zhang, and J.~Liu, ``Can
  decentralized algorithms outperform centralized algorithms? a case study for
  decentralized parallel stochastic gradient descent,'' {\em Advances in neural
  information processing systems}, vol.~30, 2017.

\bibitem{xin2020general}
R.~Xin, S.~Pu, A.~Nedi{\'c}, and U.~A. Khan, ``A general framework for
  decentralized optimization with first-order methods,'' {\em Proceedings of
  the IEEE}, vol.~108, no.~11, pp.~1869--1889, 2020.

\bibitem{sundhar2010distributed}
S.~Sundhar~Ram, A.~Nedi{\'c}, and V.~V. Veeravalli, ``Distributed stochastic
  subgradient projection algorithms for convex optimization,'' {\em Journal of
  optimization theory and applications}, vol.~147, pp.~516--545, 2010.

\bibitem{koloskova2020unified}
A.~Koloskova, N.~Loizou, S.~Boreiri, M.~Jaggi, and S.~Stich, ``A unified theory
  of decentralized sgd with changing topology and local updates,'' in {\em
  International Conference on Machine Learning}, pp.~5381--5393, PMLR, 2020.

\bibitem{wang2021cooperative}
J.~Wang and G.~Joshi, ``Cooperative sgd: A unified framework for the design and
  analysis of local-update sgd algorithms,'' {\em Journal of Machine Learning
  Research}, vol.~22, no.~213, pp.~1--50, 2021.

\bibitem{yuan2018variance}
K.~Yuan, B.~Ying, J.~Liu, and A.~H. Sayed, ``Variance-reduced stochastic
  learning by networked agents under random reshuffling,'' {\em IEEE
  Transactions on Signal Processing}, vol.~67, no.~2, pp.~351--366, 2018.

\bibitem{di2016next}
P.~Di~Lorenzo and G.~Scutari, ``Next: In-network nonconvex optimization,'' {\em
  IEEE Transactions on Signal and Information Processing over Networks},
  vol.~2, no.~2, pp.~120--136, 2016.

\bibitem{pu2021distributed}
S.~Pu and A.~Nedi{\'c}, ``Distributed stochastic gradient tracking methods,''
  {\em Mathematical Programming}, vol.~187, no.~1, pp.~409--457, 2021.

\bibitem{vaswani2017attention}
A.~Vaswani, N.~Shazeer, N.~Parmar, J.~Uszkoreit, L.~Jones, A.~N. Gomez,
  {\L}.~Kaiser, and I.~Polosukhin, ``Attention is all you need,'' {\em Advances
  in neural information processing systems}, vol.~30, 2017.

\bibitem{nair2022fundamentals}
J.~Nair, A.~Wierman, and B.~Zwart, {\em The fundamentals of heavy tails:
  Properties, emergence, and estimation}, vol.~53.
\newblock Cambridge University Press, 2022.

\bibitem{simsekli2019tail}
U.~Simsekli, L.~Sagun, and M.~Gurbuzbalaban, ``A tail-index analysis of
  stochastic gradient noise in deep neural networks,'' in {\em International
  Conference on Machine Learning}, pp.~5827--5837, PMLR, 2019.

\bibitem{zhang2020adaptive}
J.~Zhang, S.~P. Karimireddy, A.~Veit, S.~Kim, S.~Reddi, S.~Kumar, and S.~Sra,
  ``Why are adaptive methods good for attention models?,'' {\em Advances in
  Neural Information Processing Systems}, vol.~33, pp.~15383--15393, 2020.

\bibitem{gorbunov2020stochastic}
E.~Gorbunov, M.~Danilova, and A.~Gasnikov, ``Stochastic optimization with
  heavy-tailed noise via accelerated gradient clipping,'' {\em Advances in
  Neural Information Processing Systems}, vol.~33, pp.~15042--15053, 2020.

\bibitem{gurbuzbalaban2021heavy}
M.~Gurbuzbalaban, U.~Simsekli, and L.~Zhu, ``The heavy-tail phenomenon in
  sgd,'' in {\em International Conference on Machine Learning}, pp.~3964--3975,
  PMLR, 2021.

\bibitem{ahn2024lineartransformer}
K.~Ahn, X.~Cheng, M.~Song, C.~Yun, A.~Jadbabaie, and S.~Sra, ``Linear attention
  is (maybe) all you need (to understand transformer optimization),'' in {\em
  The Twelfth International Conference on Learning Representations}, 2024.

\bibitem{kunstner2024heavy}
F.~Kunstner, A.~Milligan, R.~Yadav, M.~Schmidt, and A.~Bietti, ``Heavy-tailed
  class imbalance and why adam outperforms gradient descent on language
  models,'' {\em Advances in Neural Information Processing Systems}, vol.~37,
  pp.~30106--30148, 2024.

\bibitem{charles2021large}
Z.~Charles, Z.~Garrett, Z.~Huo, S.~Shmulyian, and V.~Smith, ``On large-cohort
  training for federated learning,'' {\em Advances in neural information
  processing systems}, vol.~34, pp.~20461--20475, 2021.

\bibitem{yang2022taming}
H.~Yang, P.~Qiu, and J.~Liu, ``Taming fat-tailed (“heavier-tailed” with
  potentially infinite variance) noise in federated learning,'' {\em Advances
  in Neural Information Processing Systems}, vol.~35, pp.~17017--17029, 2022.

\bibitem{sadiev2023high}
A.~Sadiev, M.~Danilova, E.~Gorbunov, S.~Horv{\'a}th, G.~Gidel, P.~Dvurechensky,
  A.~Gasnikov, and P.~Richt{\'a}rik, ``High-probability bounds for stochastic
  optimization and variational inequalities: the case of unbounded variance,''
  in {\em International Conference on Machine Learning}, pp.~29563--29648,
  PMLR, 2023.

\bibitem{compagnoni2025adaptive}
E.~M. Compagnoni, T.~Liu, R.~Islamov, F.~N. Proske, A.~Orvieto, and A.~Lucchi,
  ``Adaptive methods through the lens of {SDE}s: Theoretical insights on the
  role of noise,'' in {\em The Thirteenth International Conference on Learning
  Representations}, 2025.

\bibitem{hubler2024gradient}
F.~H{\"u}bler, I.~Fatkhullin, and N.~He, ``From gradient clipping to
  normalization for heavy tailed sgd,'' {\em arXiv preprint arXiv:2410.13849},
  2024.

\bibitem{liu2025nonconvex}
Z.~Liu and Z.~Zhou, ``Nonconvex stochastic optimization under heavy-tailed
  noises: Optimal convergence without gradient clipping,'' in {\em The
  Thirteenth International Conference on Learning Representations}, 2025.

\bibitem{armacki2025high}
A.~Armacki, S.~Yu, P.~Sharma, G.~Joshi, D.~Bajovic, D.~Jakovetic, and S.~Kar,
  ``High-probability convergence bounds for online nonlinear stochastic
  gradient descent under heavy-tailed noise,'' in {\em The 28th International
  Conference on Artificial Intelligence and Statistics}, 2025.

\bibitem{sun2024distributed}
C.~Sun and B.~Chen, ``Distributed stochastic strongly convex optimization under
  heavy-tailed noises,'' in {\em 2024 IEEE International Conference on
  Cybernetics and Intelligent Systems (CIS) and IEEE International Conference
  on Robotics, Automation and Mechatronics (RAM)}, pp.~150--155, IEEE, 2024.

\bibitem{yu2023smoothed}
S.~Yu, D.~Jakovetic, and S.~Kar, ``Smoothed gradient clipping and error
  feedback for decentralized optimization under symmetric heavy-tailed noise,''
  {\em arXiv preprint arXiv:2310.16920}, 2023.

\bibitem{elliott2016multi30k}
D.~Elliott, S.~Frank, K.~Sima'an, and L.~Specia, ``Multi30k: Multilingual
  english-german image descriptions,'' {\em arXiv preprint arXiv:1605.00459},
  2016.

\bibitem{barsbey2021heavy}
M.~Barsbey, M.~Sefidgaran, M.~A. Erdogdu, G.~Richard, and U.~Simsekli, ``Heavy
  tails in sgd and compressibility of overparametrized neural networks,'' {\em
  Advances in neural information processing systems}, vol.~34,
  pp.~29364--29378, 2021.

\bibitem{battash2024revisiting}
B.~Battash, L.~Wolf, and O.~Lindenbaum, ``Revisiting the noise model of
  stochastic gradient descent,'' in {\em International Conference on Artificial
  Intelligence and Statistics}, pp.~4780--4788, PMLR, 2024.

\bibitem{peluchetti2020stable}
S.~Peluchetti, S.~Favaro, and S.~Fortini, ``Stable behaviour of infinitely wide
  deep neural networks,'' in {\em International Conference on Artificial
  Intelligence and Statistics}, pp.~1137--1146, PMLR, 2020.

\bibitem{lee2025efficient}
S.~H. Lee, M.~Zaheer, and T.~Li, ``Efficient distributed optimization under
  heavy-tailed noise,'' {\em arXiv preprint arXiv:2502.04164}, 2025.

\bibitem{liu2023breaking}
Z.~Liu, J.~Zhang, and Z.~Zhou, ``Breaking the lower bound with (little)
  structure: Acceleration in non-convex stochastic optimization with
  heavy-tailed noise,'' in {\em The Thirty Sixth Annual Conference on Learning
  Theory}, pp.~2266--2290, PMLR, 2023.

\bibitem{nguyen2023improved}
T.~D. Nguyen, T.~H. Nguyen, A.~Ene, and H.~Nguyen, ``Improved convergence in
  high probability of clipped gradient methods with heavy tailed noise,'' {\em
  Advances in Neural Information Processing Systems}, vol.~36,
  pp.~24191--24222, 2023.

\bibitem{chezhegov2024gradient}
S.~Chezhegov, Y.~Klyukin, A.~Semenov, A.~Beznosikov, A.~Gasnikov,
  S.~Horv{\'a}th, M.~Tak{\'a}{\v{c}}, and E.~Gorbunov, ``Gradient clipping
  improves adagrad when the noise is heavy-tailed,'' {\em arXiv preprint
  arXiv:2406.04443}, 2024.

\bibitem{duchi2011adaptive}
J.~Duchi, E.~Hazan, and Y.~Singer, ``Adaptive subgradient methods for online
  learning and stochastic optimization.,'' {\em Journal of machine learning
  research}, vol.~12, no.~7, 2011.

\bibitem{kingma2014adam}
D.~P. Kingma and J.~Ba, ``Adam: A method for stochastic optimization,'' {\em
  arXiv preprint arXiv:1412.6980}, 2014.

\bibitem{sun2024gradient}
T.~Sun, X.~Liu, and K.~Yuan, ``Gradient normalization provably benefits
  nonconvex sgd under heavy-tailed noise,'' {\em arXiv preprint
  arXiv:2410.16561}, 2024.

\bibitem{jakovetic2023nonlinear}
D.~Jakoveti{\'c}, D.~Bajovi{\'c}, A.~K. Sahu, S.~Kar, N.~Milosevi{\'c}, and
  D.~Stamenkovi{\'c}, ``Nonlinear gradient mappings and stochastic
  optimization: A general framework with applications to heavy-tail noise,''
  {\em SIAM Journal on Optimization}, vol.~33, no.~2, pp.~394--423, 2023.

\bibitem{armacki2024large}
A.~Armacki, S.~Yu, D.~Bajovic, D.~Jakovetic, and S.~Kar, ``Large deviations and
  improved mean-squared error rates of nonlinear sgd: Heavy-tailed noise and
  power of symmetry,'' {\em arXiv preprint arXiv:2410.15637}, 2024.

\bibitem{gorbunov2024highprobability}
E.~Gorbunov, A.~Sadiev, M.~Danilova, S.~Horv{\'a}th, G.~Gidel, P.~Dvurechensky,
  A.~Gasnikov, and P.~Richt{\'a}rik, ``High-probability convergence for
  composite and distributed stochastic minimization and variational
  inequalities with heavy-tailed noise,'' 2024.

\bibitem{compagnoni2025unbiased}
E.~M. Compagnoni, R.~Islamov, F.~N. Proske, and A.~Lucchi, ``Unbiased and sign
  compression in distributed learning: Comparing noise resilience via {SDE}s,''
  in {\em The 28th International Conference on Artificial Intelligence and
  Statistics}, 2025.

\bibitem{lee2024efficient}
S.~H. Lee, S.~Sharma, M.~Zaheer, and T.~Li, ``Efficient adaptive federated
  optimization,'' {\em arXiv preprint arXiv:2410.18117}, 2024.

\bibitem{yu2023secure}
S.~Yu and S.~Kar, ``Secure distributed optimization under gradient attacks,''
  {\em IEEE Transactions on Signal Processing}, 2023.

\bibitem{li2025convergence}
B.~Li and Y.~Chi, ``Convergence and privacy of decentralized nonconvex
  optimization with gradient clipping and communication compression,'' {\em
  IEEE Journal of Selected Topics in Signal Processing}, 2025.

\bibitem{taheri2023generalization}
H.~Taheri and C.~Thrampoulidis, ``On generalization of decentralized learning
  with separable data,'' in {\em International Conference on Artificial
  Intelligence and Statistics}, pp.~4917--4945, PMLR, 2023.

\bibitem{radford2018improving}
A.~Radford, K.~Narasimhan, T.~Salimans, I.~Sutskever, {\em et~al.}, ``Improving
  language understanding by generative pre-training,'' 2018.

\bibitem{xin2020variance}
R.~Xin, U.~A. Khan, and S.~Kar, ``Variance-reduced decentralized stochastic
  optimization with accelerated convergence,'' {\em IEEE Transactions on Signal
  Processing}, vol.~68, pp.~6255--6271, 2020.

\bibitem{olshevsky2014linear}
A.~Olshevsky, ``Linear time average consensus on fixed graphs and implications
  for decentralized optimization and multi-agent control,'' {\em arXiv preprint
  arXiv:1411.4186}, 2014.

\bibitem{gharesifard2012distributed}
B.~Gharesifard and J.~Cort{\'e}s, ``Distributed strategies for generating
  weight-balanced and doubly stochastic digraphs,'' {\em European Journal of
  Control}, vol.~18, no.~6, pp.~539--557, 2012.

\bibitem{assran2019stochastic}
M.~Assran, N.~Loizou, N.~Ballas, and M.~Rabbat, ``Stochastic gradient push for
  distributed deep learning,'' in {\em International Conference on Machine
  Learning}, pp.~344--353, PMLR, 2019.

\bibitem{mohar1991laplacian}
B.~Mohar, Y.~Alavi, G.~Chartrand, and O.~Oellermann, ``The laplacian spectrum
  of graphs,'' {\em Graph theory, combinatorics, and applications}, vol.~2,
  no.~871-898, p.~12, 1991.

\bibitem{xin2021improved}
R.~Xin, U.~A. Khan, and S.~Kar, ``An improved convergence analysis for
  decentralized online stochastic non-convex optimization,'' {\em IEEE
  Transactions on Signal Processing}, vol.~69, pp.~1842--1858, 2021.

\bibitem{nedic2009distributed}
A.~Nedic and A.~Ozdaglar, ``Distributed subgradient methods for multi-agent
  optimization,'' {\em IEEE Transactions on Automatic Control}, vol.~54, no.~1,
  pp.~48--61, 2009.

\bibitem{xiao2005scheme}
L.~Xiao, S.~Boyd, and S.~Lall, ``A scheme for robust distributed sensor fusion
  based on average consensus,'' in {\em IPSN 2005. Fourth International
  Symposium on Information Processing in Sensor Networks, 2005.}, pp.~63--70,
  IEEE, 2005.

\bibitem{beaton1974fitting}
A.~E. Beaton and J.~W. Tukey, ``The fitting of power series, meaning
  polynomials, illustrated on band-spectroscopic data,'' {\em Technometrics},
  vol.~16, no.~2, pp.~147--185, 1974.

\bibitem{lin2021quasi}
T.~Lin, S.~P. Karimireddy, S.~Stich, and M.~Jaggi, ``Quasi-global momentum:
  Accelerating decentralized deep learning on heterogeneous data,'' in {\em
  International Conference on Machine Learning}, pp.~6654--6665, PMLR, 2021.

\bibitem{carnevale2022gtadam}
G.~Carnevale, F.~Farina, I.~Notarnicola, and G.~Notarstefano, ``Gtadam:
  Gradient tracking with adaptive momentum for distributed online
  optimization,'' {\em IEEE Transactions on Control of Network Systems},
  vol.~10, no.~3, pp.~1436--1448, 2022.

\bibitem{shulgin2025smoothed}
E.~Shulgin, S.~Khirirat, and P.~Richt{\'a}rik, ``Smoothed normalization for
  efficient distributed private optimization,'' {\em arXiv preprint
  arXiv:2502.13482}, 2025.

\end{thebibliography}
